\documentclass[final,times]{elsarticle}

 \makeatletter
\def\ps@pprintTitle{%
 \let\@oddhead\@empty
 \let\@evenhead\@empty
\def\@oddfoot{\centerline{\thepage}}%
 \let\@evenfoot\@oddfoot}
\makeatother

 \newcommand{\grad}{\triangledown}

\usepackage{graphicx,enumitem,dsfont}
 \usepackage[dvips]{epsfig}
\usepackage{amscd}
\usepackage{amssymb}
\usepackage{amsthm}
\usepackage{amsmath}
\usepackage{latexsym}

\usepackage{upref}
\usepackage{hyperref}
\usepackage{color}
\theoremstyle{plain}
\newtheorem{thm}{Theorem}[section]
\theoremstyle{plain}
\newtheorem{lem}[thm]{Lemma}
\newtheorem{prop}[thm]{Proposition}

\theoremstyle{definition}
\newtheorem{defi}{Definition}[section]
\newtheorem{rem}{Remark}[section]
\newtheorem*{maintheorem*}{Main Theorem}
\newtheorem*{maincorollary*}{Main Corollary}
\newenvironment{Assumptions}
{%
\setcounter{enumi}{0}

\begin{enumerate}}%
{\end{enumerate} }
\newenvironment{Assumptions2}
{
\setcounter{enumi}{0}

\begin{enumerate}}
{\end{enumerate} }

{%
\setcounter{enumi}{0}

\begin{enumerate}}%
{\end{enumerate} }

\newcommand{\norm}[1]{\ensuremath{\left\|#1\right\|}}
\newcommand{\abs}[1]{\ensuremath{\left|#1\right|}}

\newcommand{\goto}{\ensuremath{\rightarrow}}

\newcommand{\eps}{\ensuremath{\epsilon}}
\newcommand{\R}{\ensuremath{\mathbb{R}}}

\newcommand{\rd}{\ensuremath{\R^d}}

\newcommand{\supp}{\ensuremath{\mathrm{supp}\,}}

\numberwithin{equation}{section} \allowdisplaybreaks

\journal{preprint}
\begin{document}

\begin{frontmatter}

\title{Continuous dependence estimate for conservation laws with L\'{e}vy noise  }

\author[Imran H. Biswas]{Imran H. Biswas\corref{cor}}
\address[Imran H. Biswas]{
 Centre for Applicable Mathematics,
 Tata Instiute of Fundamental Research,
  P.O.\ Box 6503,
  Bangalore 560065, India}
    \ead{imran@math.tifrbng.res.in}
\cortext[cor]{Corresponding author.}
\author[Imran H. Biswas]{Ujjwal Koley}
   \ead{ujjwal@math.tifrbng.res.in}

\author[Imran H. Biswas]{Ananta K. Majee}
   \ead{majee@math.tifrbng.res.in}

\begin{abstract}We are concerned with multidimensional stochastic balance laws driven by L\'{e}vy processes. 
 Using bounded variation (BV) estimates for vanishing viscosity approximations, we derive
 an explicit continuous dependence estimate on the nonlinearities of the 
 entropy solutions under the assumption that L\'{e}vy noise only depends on the solution. This result is used to show the error estimate for 
 the stochastic vanishing viscosity method. In addition, we establish fractional $BV$ estimate for vanishing viscosity
approximations in case the noise coefficient depends on both the solution and spatial variable.
\end{abstract}


\end{frontmatter}

\section{Introduction} The last couple of decades have witnessed remarkable advances in the studies of partial differential equations with noise/randomness. A vast literature is now available on the subject of stochastic partial differential equations (SPDEs) and the particular frontier involving hyperbolic conservation laws with noise has had its fair share of attention as well. However, this is still very much a developing story and there still a number of issues waiting to be explored. In this paper, we aim at deriving continuous dependence estimates based on nonlinearities for stochastic conservation laws driven by multiplicative L\'{e}vy noise.  A formal description of our problem requires a filtered probability space $\big(\Omega, P, \mathcal{F}, \{\mathcal{F}_t\}_{t\ge 0} \big)$ and we are interested in an $L^p(\R^d)$-valued predictable process $u(t,\cdot)$ which satisfies the Cauchy problem
 \begin{equation}
 \label{eq:levy_stochconservation_laws}
 \begin{cases} 
 du(t,x) + \mbox{div}_x F(u(t,x))\,dt
 =\int_{|z|> 0} \eta( u(t,x);z) \, \tilde{N}(dz,dt), & \quad x \in \Pi_T, \\
 u(0,x) = u_0(x), & \quad  x\in \R^d,
\end{cases}
\end{equation}
where $\Pi_T=(\R^d \times (0,T))$ with $T>0$ fixed. The initial condition $u_0(x)$ is a given function on $\R^d$, and $F:\R \mapsto \R^d$ is given (sufficiently smooth) vector valued flux function (see Section~\ref{tech-frame-main-result} for the complete list of assumptions).  The right hand side of \eqref{eq:levy_stochconservation_laws} represents the noise term and it is composed of a compensated Poisson random measure $ \tilde{N}(dz,dt)= N(dz,dt)-\nu(dz)\,dt $, where $N$ is a Poisson random 
measure on $\R\times (0,\infty)$  with intensity measure $\nu(dz)$, and the jump amplitude (integrand) $\eta(u, z)$ is real valued function signifying the multiplicative nature of the noise.

   Hyperbolic conservation laws are used to describe a large number of physical phenomenon from areas such as physics, economics, biology etc. The inherent uncertainty in such phenomenon prompts one to account for the same and consider random perturbation of conservation laws. As an important first step into the subject, a significant body of literature has grown around conservation laws that are perturbed by Brownian white noise.  However, due to the complex nature of the uncertainties,  it is only natural to look beyond Brownian white noise settings and consider problems with more general type of noise. We do that in this paper in the problem \eqref{eq:levy_stochconservation_laws} by introducing Poisson noise in the right hand side. It is also mentioned the result of this paper could be extended to the general L\'{e}vy noise case.

In the case $\eta=0$, the equation \eqref{eq:levy_stochconservation_laws} becomes a standard conservation laws in $\R^d$.
For the deterministic conservation laws, well-posedness analysis has a very long tradition and it goes back to the $1950$s. 
However, we will not be able to discuss the whole literature here, but only refer to the parts that are pertinent to the current paper.  
The question of existence and uniqueness of solutions of conservation laws was first settled in the pioneer papers of  Kru\v{z}kov \cite{kruzkov} and Vol'pert \cite{Volpert}. For a completely satisfactory well-posedness theory of conservation laws, 
we refer to the monograph of Dafermos \cite{dafermos}. See also \cite{godu} and references therein. 

\subsection{Stochastic balance laws driven by Brownian white noise }    

As has been mentioned, evolutionary SPDEs with L\'{e}vy noise has been the topic of interest of many authors lately, and new results are emerging faster than ever before. However,  the study of stochastic balance laws driven 
by noise has so far been limited to equations that are driven by Brownian white noise  and a satisfactory well-posedness theory is
available by now. 

Observe that when the noise is of additive nature, a change of variable reduces equation 
into a hyperbolic conservation law with random flux which could be analyzed with deterministic 
techniques. In fact, Kim \cite{KIm2005} extended Kru\v{z}kov's entropy formulation to establish
the well-posedness of one dimensional stochastic balance law.

However, when the noise is of multiplicative nature, one could not apply a straightforward 
Kru\v{z}kov's doubling method to get a $L^1$-contraction principle as in \cite{kruzkov}. The main difficulty lies in doubling 
the \emph{time} variable which gives rise to stochastic integrands that are anticipative and hence the stochastic integrals in the sense 
of It\^{o}-L\'{e}vy would not make sense. Hence, it fails to capture a specific ``noise-noise'' interaction term relating two 
entropy solutions. This issue was first resolved by Feng $\&$ Nualart \cite{nualart:2008} with the introduction of additional 
condition, which captures the missing ``noise-noise'' interaction term, the so called  {\em strong stochastic entropy solution}. They used $L^p$ framework to prove the 
multidimensional uniqueness result for strong stochastic entropy solution.  However, existence was restricted to \emph{one space dimension}
since their proof of existence  was based on a stochastic version of  \emph{compensated compactness} argument
applied to vanishing viscosity approximation of the underlying problem. To overcome this problem, Debussche $\&$ Vovelle \cite{Vovelle2010} introduced
kinetic formulation of such problems and as a result they were able to established the wellposedness
of multidimensional stochastic balance law via kinetic approach. At around the same time, Chen $\&$ Karlsen \cite{Chen-karlsen2012} also
established multidimensional wellposedness of strong entropy solution in $L^p \cap BV$, via $BV$ framework. Moreover, they were able to
develop continuous dependence theory for multidimensional balance laws and, as a by product, they derived an explicit \emph{convergence
rate} of the approximate solutions to the underlying problem. We also mention that, using the concept of measure valued solutions and Kru\v{z}kov's semi-entropy 
formulations, a result of existence and uniqueness of the entropy solution has been obtained by Bauzet. et. al. in \cite{vallet}. 

In the article \cite{nualart:2008},
the authors used an entropy formulation which is  strong in time but weak in space, which is in our view may give rise to problems
where the solutions are not shown to have continuous sample paths. We refer to \cite{biswas-majee 2013}, where a few technical
questions are raised and remedial measures have been proposed. We also mention that Weinen et. al. \cite{Sinai1997} published a very influential article describing the existence, uniqueness
and weak convergence of invariant measures for one dimensional Burger's equation with stochastic forcing which is periodic in $x$.

\subsection{Stochastic balance laws driven by L\'{e}vy noise}    
Despite relatively large body of research on stochastic partial differential equations that are driven by
L\'{e}vy noise, to the best of our knowledge, very little is available on the specific problem of 
conservation laws with L\'{e}vy noise. In fact, the first attempt were made
to build a comprehensive theory on such problems in a very recent article by Biswas. et. al. \cite{kbm2014}.
For a detailed introduction to the SPDEs driven by L\'{e}vy processes, we refer
to the monograph by Peszat. et.al. \cite{peszat} and references therein. 
Roughly speaking, the theory developed in \cite{peszat} covers semi linear 
parabolic equations driven by L\'{e}vy noise, which could be treated as stochastic evolution equations in some infinite 
dimensional Banach or Hilbert space, and typically the solutions of such equations enjoy regularizing properties. However, we can't emulate those
techniques on the specific problem of conservation laws driven by  L\'{e}vy noise due to the intrinsic discontinuous 
nature of the solution.

In fact, independent of the smoothness of the initial data $u_0(x)$, due to the presence of nonlinear flux term  in
equation \eqref{eq:levy_stochconservation_laws}, 
solutions to \eqref{eq:levy_stochconservation_laws} are not necessarily smooth and weak solutions must be sought. 
Before introducing the concept of weak solutions, we first assume that 
 the  filtered probability space $\big(\Omega, P, \mathcal{F}, \{\mathcal{F}_t\}_{t\ge 0} \big)$ satisfies 
the usual hypothesis, i.e.,  $\{\mathcal{F}_t\}_{t\ge 0}$ is a right-continuous filtration such that $\mathcal{F}_0$ 
contains all the $P$-null subsets of $(\Omega, \mathcal{F})$. Moreover,
by a predictable $\sigma$-field on $[0,T]\times \Omega$, denoted
by $\mathcal{P}_T$, we mean that the 
$\sigma$-field generated by the sets of the form: $ \{0\}\times A$ and $(s,t]\times B$  for any $A \in \mathcal{F}_0; B
\in \mathcal{F}_s,\,\, 0<s,t \le T$.

The notion of weak solution is defined as follows:
\begin{defi} [weak solution]
\label{defi:stochweaksol}
An $ L^2(\R^d)$-valued $\{\mathcal{F}_t: t\geq 0 \}$-predictable stochastic process $u(t)= u(t,x)$ is
called a stochastic weak solution of \eqref{eq:levy_stochconservation_laws} if for all non-negative test functions $\psi \in C_c^{\infty}([0,T) \times \R^d)$,
\begin{align}
\int_{\R^d} \psi(0,x) u(0,x)\,dx & + \int_{\R^d} \int_{0}^T \Big\{ \partial_t \psi(t,x) u(t,x)
+  F(u(t,x)) \cdot \nabla_x \psi(t,x) \Big\}\,dx\,dt \notag \\ 
&\qquad + \int_{t=0}^T\int_{|z|>0}\int_{\R^d}  \eta( u(t,x);z) \psi(t,x) \,dx\,\tilde{N}(dz,dt) = 0, \quad P-\text{a.s}.
\label{3w1}
\end{align}
\end{defi} 

However, it is well known that weak solutions may be discontinuous and they
are not uniquely determined by their initial data. Consequently, an entropy condition must be
imposed to single out the physically correct solution. Since the notion of entropy solution is built 
around the so called entropy-entropy flux pairs, we begin with the definition of  entropy-entropy flux pairs.
\begin{defi}[entropy-entropy fux pair]
An ordered pair $(\beta,\zeta) $ is called an entropy-entropy flux pair if $ \beta \in C^2(\R) $ with $\beta \ge0$, and $\zeta = (\zeta_1,\zeta_2,....\zeta_d):\R \mapsto\rd $ is a vector field satisfying
\begin{align*}
\zeta'(r) = \beta'(r)F'(r), \,\,\text{for all r}.
\end{align*}
Moreover, an entropy-entropy flux pair $(\beta,\zeta)$ is called convex if $ \beta^{\prime\prime}(\cdot) \ge 0$.  
\end{defi}
With the help of a convex entropy-entropy flux pair $(\beta,\zeta)$, the notion of stochastic entropy solution is defined as follows:
\begin{defi} [stochastic entropy solution]
\label{defi:stochentropsol}
An $ L^2(\rd)$-valued $\{\mathcal{F}_t: t\geq 0 \}$-predictable stochastic process $u(t)= u(t,x)$ is
called a stochastic entropy solution of \eqref{eq:levy_stochconservation_laws} provided \\
(1) For each $ T>0$, $p=2,3,4,\cdots,$ \[\sup_{0\leq t\leq T} E\Big[||u(t, \cdot)||_{p}^{p} \Big] < \infty. \] 
(2) For all test functions $0\leq \psi\in C_{c}^{1,2}([0,\infty )\times\rd) $, and each convex entropy pair 
$(\beta,\zeta) $, 
\begin{align}
&  \int_{\R_x^d} \psi(0,x) \beta(u(0,x))\,dx + \int_{\Pi_T} \Big\{ \partial_t \psi(t,x) \beta(u(t,x))
+  \zeta(u(t,x)) \cdot \nabla_x \psi(t,x) \Big\}\,dx\,dt \notag \\ 
 +& \int_{r=0}^T\int_{|z|>0}\int_{\R_x^d} \Big(\beta\big(u(r,x) +\eta(u(r,x);z)\big)- \beta(u(r,x)\Big)\psi(r,x)\,dx\,\tilde{N}(dz,dr) \notag \\
 + & \int_{\Pi_T}\int_{|z|>0} \Big(\beta\big(u(r,x) +\eta(u(r,x);z)\big)- \beta(u(r,x))-\eta(u(r,x);z)\beta^\prime(u(r,x)) \Big)\psi(r,x)\,\nu(dz)\,dr\,dx \notag \\
  & \ge 0 \quad P-\text{a.s.} \notag
\end{align}
\end{defi} 

Due to the nonlocal nature of the entropy ineaualities  and the noise-noise interaction, the Definition~\ref{defi:stochentropsol} 
alone does not seem to  give the $L^1$-contraction principle in the sense of average and hence the uniqueness is not immediate. In other words, classical 
``doubling of variable'' technique in time variable does not work when one tries to compare directly two entropy solutions defined in the sense of 
Definion~\ref{defi:stochentropsol}. To overcome this problem, the authors in \cite{vallet,kbm2014} used a more direct approach
by comparing one entropy solution against the solution of the regularized problem and subsequently sending the regularized 
parameter to zero, relying on ``weak compactness'' of the regularized approximations.

In order to successfully implement the direct approach, one needs to weaken the 
notion of stochastic entropy solution, and subsequently install the notion of so called generalized entropy solution (cf. \cite{vallet,kbm2014}).
\begin{defi} [generalized entropy solution]
\label{defi: young_stochentropsol}
An $ L^2\big(\R^d \times (0,1)\big)$-valued $\{\mathcal{F}_t: t\geq 0 \}$-predictable stochastic process $v(t)= v(t,x,\alpha)$ is
called a generalized stochastic entropy solution of \eqref{eq:levy_stochconservation_laws} provided \\
(1) For each $ T>0$, $ p=2,3,4, \cdots,$ \[\sup_{0\leq t\leq T} E\Big[||v(t,\cdot,\cdot)||_{p}^{p}\Big] <
\infty. \] 
(2) For all test functions $0\leq \psi\in C_{c}^{1,2}([0,\infty )\times\R^d)$, and each convex entropy pair 
$(\beta,\zeta) $, 
\begin{align}
&  \int_{\R_x^d} \psi(0,x)\beta(v(0,x))\,dx + \int_{\Pi_T} \int_{\alpha=0}^1  \Big( \partial_t \psi(t,x)\beta(v(t,x,\alpha)) + 
 \zeta(v(t,x,\alpha))\cdot\nabla_x \psi(t,x) \Big) \,d\alpha\,dx\,dt \notag \\ 
 +&  \int_{r=0}^T\int_{|z|>0}  \int_{\R_x^d} \int_{\alpha=0}^1 \Big( \beta \big(v(r,x,\alpha) +\eta(v(r,x,\alpha);z)\big)- \beta(v(r,x,\alpha))\Big)
 \psi(r,x)\,d\alpha \tilde{N}(dz,dr)\,dx \notag\\
 + & \int_{r=0}^T \int_{|z|>0} \int_{\R_x^d} \int_{\alpha=0} ^1  \Big( \beta \big(v(r,x,\alpha) +\eta(v(r,x,\alpha);z)\big)-
\beta(v(r,x,\alpha))-\eta(v(r,x,\alpha);z) \beta^{\prime}(v(r,x,\alpha)) \Big) \notag \\
& \hspace{4cm} \times \psi(r,x)\,d\alpha\,dx \, \nu(dz)\,dr\notag \\
 \ge  &  0  \quad P-\text{a.s.} \notag 
\end{align}
\end{defi}
As we mentioned earlier,  in a recent article \cite{kbm2014},  the authors established well-posedness 
along with few a priori estimates for the viscous problem 
with L\'{e}vy noise and proved the existence and uniqueness of generalized entropy solution 
for multidimensional Cauchy problem \eqref{eq:levy_stochconservation_laws} via Young measure approach.
Finally, we mention that Dong and Xu \cite{xu} established the global 
well-posedness of strong, weak and mild solutions for one-dimensional viscous Burger's equation driven by Poisson process with 
Dirichlet boundary condition via Galerkin method. Also, they proved the existence of invariant measure of the solution.

\subsection{Scope and outline of this paper} 
The above discussions clearly highlights the lack of stability estimates for the entropy solutions of
stochastic balance laws driven by L\'{e}vy noise. In this paper, drawing preliminary motivation
from \cite{Chen-karlsen2012}, we intend to 
develop a continuous dependence theory for stochastic entropy solution which in turn can be used to
derive an error estimate for the vanishing viscosity method. However,  it seems difficult to develop such a theory
without securing a BV estimate for stochastic entropy solution. 
As a result, we first address the question of existence, uniqueness of stochatic BV- entropy solution in 
$L^p(\R^d) \cap BV (\R^d)$ of the problem \eqref{eq:levy_stochconservation_laws}. Making use of 
the crutial BV estimate, we provide a continuous depenece estimate and error estimate for the vanishing viscosity method 
provided initial data lies in $u_0 \in L^p(\R^d) \cap BV(\R^d)$. 

Finally, we turn our discussions to more general stochastic balance laws driven by  L\'{e}vy processes, namely
when the function $\eta$ in the L\'{e}vy noise term  has explicit dependency on the spatial position $x$ as well.  
In view of the discussions in \cite{ Chen-karlsen2012}, in this case we can't expect BV estimates, but instead 
a fractional BV estimate is expected. However, that does not prevent us to provide an existence proof
for more general class of equations in $L^p(\R^d)$.

The remaining  part of this paper is organized  as follows: we collect all the assumptions needed in the subsequent analysis,
results for the regularized problem and finally state
the main results in Section \ref{tech-frame-main-result}. In Section~\ref{sec:apriori}, 
we prove uniform spatial BV estimate for the solution of vanishing
viscosity approximation of \eqref{eq:levy_stochconservation_laws}, and thereby establishing $BV$ bounds for entropy solutions. Section 
 \ref{cont-depen-estimate} deals with the continuous dependence estimate, while Section~\ref{sec:error_estimate} deals with the error estimate.  
Finally, in Section~\ref{sec:frac}, we establish a fractional $BV$ estimate for a larger class of stochastic balance laws.
      
\section{Preliminaries}
\label{tech-frame-main-result}
We mention that, throughout this paper we use $C, K$ to denote a generic constants; the actual values
of $C, K$ may change from one line to the next during a calcuation.
The Euclidean norm on any $\mathbb{R}^d$-type space is
denoted by $|\cdot|$ and the norm in $BV(\R^d)$ is denoted by $|\cdot|_{BV(\rd)}$.      

Next, we collect all the basic assumptions on the data of the problem \eqref{eq:levy_stochconservation_laws}.
\begin{Assumptions}
 \item \label{A1}  The initial function $u_0(x)$ is a  $\cap_{p=1,2,..} L^p(\rd)$-valued  $\mathcal{F}_0$-measurable random variable  satisfying
\[E \Big[||u_0||_{p}^{p} + ||u_0||_{2}^{p} + |u_0|_{BV(\R^d)}\Big] < \infty  \qquad\text{for}~ p=1,2,...~ .\]

 \item \label{A2} For every $ k = 1,2...,d $, the functions $F_k(s) \in C^2(\R) $, and $F_k(s), F_k^\prime(s) $ and $F_k^{\prime\prime}(s)$ 
 have at most polynomial growth in $ s$. 
 
\item \label{A3} There exist  positive constants   $0<\lambda^* <1$  and  $ C>0$,  such that 
$ ~\text{for all}~  u,v \in \R;~~z\in \R$
 \begin{align*} | \eta(u;z)-\eta(v;z)|  & \leq  \lambda^* |u-v|( |z|\wedge 1) \\
  \text{and}\quad|\eta(u;z)| & \le C(1+|u|)(|z|\wedge 1).
 \end{align*}
 

\item \label{A4'}  To prove existence and uniqueness of solutions, we assume that the L\'{e}vy measure $\nu(dz)$ which has a possible singularity at $z=0$, satisfies 
\begin{align*}
 \int_{|z|>0} (1\wedge |z|^2)\,\nu(dz) <  + \infty.
 \end{align*}
    
\end{Assumptions}
\begin{rem}
 Note that we need the assumption \ref{A2} as a result of the requirement that the 
 entropy solutions satisfy $L^p$ bounds for all $p\ge 2$, which in turn forces us to choose initial data satisfying ~\ref{A1}.
 However, it is possible to get entropy solution for initial data 
in $L^2(\R^d) \cap BV(\R^d)$, provided the  given flux function is globally Lipschitz. The assumption ~\ref{A3} is natural in the context of L\'{e}vy noise with the exception of $\lambda^*\in (0,1)$, which is necessary for the uniqueness.
Finally, the assumptions \ref{A1}-\ref{A4'} collectively ensures existence and uniqueness of stochastic entropy solution,  and the continuous dependence estimate as well. \end{rem}
To this end, for any given fixed $\epsilon > 0$, we consider the viscous perturbation of \eqref{eq:levy_stochconservation_laws}
  \begin{equation}
  \label{eq:levy_stochconservation_laws-viscous}
  \begin{split}
  du_\eps(t,x) + \mbox{div}_x F_\eps(u_\eps(t,x)) \,dt &=  \int_{|z|> 0} \eta_\eps(u_\eps(t,x);z)\,\tilde{N}(dz,\,dt)+ \eps \Delta_{xx} u_\eps \,dt,\,\, t>0, ~ x\in \R^d, \\
  u(0,x) &= u_\eps(0,x),  \,\, x \in \R^d,
\end{split}
\end{equation}
where $u_\eps(0,x)$ is a smooth approximation of initial data $u_0(x)$ such that
\begin{align}
 E\Big[\int_{\R_x^d}|u_\eps(0,x)|^p\,dx\Big] \le E\Big[\int_{\R_x^d}|u_0(x)|^p\,dx\Big].
\end{align} 
Moreover, if initial data $u_0(x) \in BV(\R^d)$, then 
\begin{align}
 E\Big[\int_{\R_x^d}| \grad u_\eps(0,x)|\,dx\Big] \le E\Big[\int_{\R_x^d}| \grad u_0(x)|\,dx\Big].
\end{align}
Furthermore, mainly to ease the presentation throught this paper, we assume that $F_\eps, \eta_{\eps}$ are ``sufficiently smooth'' approximations 
of $F$ and $\eta$ respectively. More specifically, we require that $F_\eps$ and $\eta_\eps$ satisfy the same properties as $F$ and $\eta$ respectively (cf.~\ref{A2}--~\ref{A3}) and
\begin{align}
 |F_\eps(r)-F(r)| &\le C\eps(1+|r|^{p_0}),\,\, \text{for some}\,\, p_0 \in \mathbb{N}, \notag\\
  |\eta_\eps(u;z) -\eta(u;z)| &\le C\eps(1+|u|) (1\wedge |z|).\label{eq:regularization}
\end{align} 
Observe that, in view of \cite[Subsection $3.2$]{kbm2014}, these properties of $F_{\eps}$ and $\eta_{\eps}$ are justified.

For the deterministic counterpart of \eqref{eq:levy_stochconservation_laws-viscous}, proof of existence of global smooth solutions is classical by now. Same techniques could be used, mutatis mutandis, also for the stochastic scenario to establish the existence. More precisely, we have the following proposition from \cite{kbm2014}.
\begin{prop} 
\label{prop:vanishing viscosity-solution}
Let the assumptions \ref{A1}, \ref{A2}, \ref{A3}, and \ref{A4'} hold and $\eps>0$ be a given positive number. Then there exists a unique $C^2(\R^d)$-valued predictable process $u_\eps(t,\cdot)$ which solves the initial value problem \eqref{eq:levy_stochconservation_laws-viscous}.
Moreover, 
\begin{itemize}
 \item [$(a)$] The solution  $u_\eps(t,x)$ satisfies, almost surely,
 \begin{align}
   u_\eps(t,x) = &\int_{\R_y^d} G(t,x-y)u_0(y) dy - \int_{s=0}^t \int_{\R_y^d} G(t-s,x-y)\grad \cdot F_{\eps}(u_\eps(s,y)) \,dy\,ds \notag\\
    & + \int_{s=0}^t \int_{|z| > 0}\int_{\R_y^d} G(t-s, x-y) \eta( u_\eps(s,y);z)\,dy\,\tilde{N}(dz,ds), \notag \end{align} 
 where $G(t,x)$ is the heat kernel associated with the operator $\eps \Delta_{xx}$ i.e., 
 \begin{align*}
 G(t,x):= G_\eps(t,x) = \frac{1}{(4\pi \eps t)^{\frac{d}{2}}} e^{\frac{-|x|^2}{4\eps t}},\quad t> 0.
 \end{align*}
 \item [$(b)$]  For  positive integer $p=1,2,3,\cdots,$ and $T>0$
 \begin{align}
   \sup_{\eps > 0} \sup_{0\le t\le T} E\Big[ ||u_\epsilon (t,\cdot)||_p^p\Big] < \infty.\label{uniform-estimates}
 \end{align}
 \item [$(c)$] For a  function $ \beta \in C^2(\R)$  with $ \beta,\beta',\beta''$ having at most
polynomial growth,
\begin{align}
\sup_{\eps >0} E\left[\Big|\eps \int_{t=0}^T \int_{\R_x^d} \beta''(u_\eps(t,x))|\grad_x
u_{\eps}(t,x)|^2\,dx\,dt\Big|^p\right] < \infty, \quad p=1,2...,\,T>0.\notag
\end{align}
\end{itemize}
\end{prop}
\begin{rem}
 \label{rem:integrability}
  In view of Proposition \ref{prop:vanishing viscosity-solution} and assumption \ref{A1}, it follows that, for each fixed
  $\eps>0$, $\grad u_\eps(t,x)$  is integrable. Moreover if $E\Big[\int_{\R_x^d} |\grad^2 u_{\eps}(0,x)|\,dx\Big] < + \infty$, then
  $\grad^2 u_{\eps}(t,x)$ is also integrable for fixed $\eps>0$ and any finite time $T>0$ (cf. \cite[Section $3$]{kbm2014}).
\end{rem}

Now we are in a position to state the main results of this article.
\begin{maintheorem*} [continuous dependence estimate]
\label{thm:continuous dependence estimate} 
  Let the assumptions \ref{A1}, ~\ref{A2}, ~\ref{A3}, and \ref{A4'} hold for two sets of 
  given data $(u_0,F,\eta)$ and $(v_0,G,\sigma)$. Let $u(t,x)$ be any entropy solution 
  of \eqref{eq:levy_stochconservation_laws} with initial data $u_0(x)$ and $v(s,y )$ be 
  another entropy solution with initial data $v_0(y)$ and satisfies
  \begin{align} 
dv(s,y) + \mbox{div}_y G(v(s,y))\,ds &= \int_{|z|> 0} \sigma (v(s,y);z) \, \tilde{N}(dz,ds). \label{eq:stability-12}
\end{align} In addition, we assume that 
$F^{\prime\prime},\, F^\prime-G^\prime\in L^\infty $and define $\mathcal{D}(\eta, \sigma) :=\displaystyle{ \sup_{u\in \R}}\int_{|z|> 0} \frac{\big(\eta(u;z) -\sigma(u;z)\big)^2}{1+|u|^2} \nu(dz)$. Then there exists a constant $C_T>0$, independent of $|u_0|_{BV(\R^d)}$ and
$|v_0|_{BV(\R^d)}$, such that for a.e.~~ $t\ge 0$, 
\begin{align}
&E \Big[\int_{\R_x^d}   \big|u(t,x)  -v(t,x)\big|\phi(x) \,dx\Big] \notag \\
 & \le  C_T \,\Bigg[
 \big( 1+ E[|v_0|_{BV(\R^d)}]\big)\sqrt{t \mathcal{D}(\eta, \sigma)}||\phi(\cdot)||_{L^{\infty}(\R^d)} 
+ E \big[|v_0|_{BV(\R^d)}\big]\, ||F^\prime-G^\prime||_{\infty}\, t\, ||\phi(\cdot)||_{L^{\infty}(\R^d)}\notag \\
 &\qquad\qquad +   E \Big[\int_{\R_x^d}| u_0(x) -v_0(x)| \phi(x)\,dx  \Big]
+\sqrt{t \mathcal{D}(\eta, \sigma)} ||\phi(\cdot)||_{L^1(\R^d)}\Bigg],
 \end{align}  where $0\le \phi \in C_c^2(\R^d)$ such that $|\grad \phi(x)| \le C \phi(x)$ and $|\Delta \phi(x)| \le C \phi(x)$
 for some constant $C>0$.
 Moreover,  a special choice of $\phi(x)$ with the above properties
 \begin{align*}
 \phi(x) =\begin{cases} 
 1,\quad &\text{when} ~ \abs{x} \le R,\\
 e^{-C \big(\abs{x} -R \big)}, \quad &\text{when}~ \abs{x} \ge R,
                 \end{cases}
 \end{align*}
leads to the following simplified result: For any $R >0$, there exists a constant $C_T^R>0$, independent of $|u_0|_{BV(\R^d)}$ and
$|v_0|_{BV(\R^d)}$, such that for a.e.~~ $t\ge 0$, 
\begin{align}
&E \Big[\int_{\abs{x} \le R}   \big|u(t,x)  -v(t,x)\big|\,dx\Big] \notag \\
 \le & C_T^R\,\Bigg[
 \big( 1+ E[|v_0|_{BV(\R^d)}]\big) \sqrt{t \mathcal{D}(\eta, \sigma)}+ t\, E \big[|v_0|_{BV(\R^d)}\big]\, ||F^\prime-G^\prime||_{\infty} +   E \Big[\int_{\R_x^d}| u_0(x) -v_0(x)| \,dx \Big]
\Bigg].
 \end{align}
\end{maintheorem*}
\begin{rem}
 The condition that $F^{\prime\prime},\, F^\prime-G^\prime \in L^\infty $ could be avoided if we assume that $u,v \in L^\infty((0,T)\times \R^d \times \Omega)$ for any time $T>0$. In this case, an appropriate version of the main theorem would be possible. Moreover, the quantity $\mathcal{D}(\eta, \sigma)$ is well defined in view of \ref{A3} and \ref{A4'}.
\end{rem}

As a by product of the above theorem, we have the following corollary:
\begin{maincorollary*}[error estimate]
\label{cor: Error estimate}
Let the assumptions \ref{A1}, ~\ref{A2}, ~\ref{A3}, ~\ref{A4'} hold and let $u(t,x)$ be any entropy 
  solution of \eqref{eq:levy_stochconservation_laws} with  $E \big[|u(t,\cdot)|_{BV(\R^d)}\big]\le E\big[|u_0|_{BV(\R^d)}\big]$, for
  $t>0$. In addition, we assume that 
  $F^{\prime\prime} \in L^\infty $.
Then, there exists a constant $C_T>0$, independent of $|u_0|_{BV(\R^d)}$, such that for a.e. $t\ge 0$
\begin{align}
 E \Big[\int_{\R_x^d} \big|u_\eps(t,x) & -u(t,x)\big|\,dx\Big] \notag \\
  & \qquad \le  C_T \Big \{ \eps^{\frac{1}{2}} \big( 1 + E[|u_0|_{BV(\R^d)}] \big)(1+ t)
     + E \big[\int_{\R_x^d} \big|u_\eps(0,x)-u_0(x)\big|\,dx\big] \Big\}.\notag
 \end{align}
  \end{maincorollary*}
\noindent Moreover, if we assume that the initial error $E \big[\int_{\R_x^d} \big|u_\eps(0,x)-u_0(x)\big|\,dx\big] = \mathcal{O}(\eps^{\frac12})$, then we get
\begin{align*}
E \Bigg[\int_{\R_x^d} \big|u_\eps(t,x) & -u(t,x)\big|\,dx\Bigg] = \mathcal{O}(\eps^{\frac12}).
\end{align*}
Here we used the notation $\mathcal{O}(\eps)$ to denote quantities that depend on $\eps$ and are bounded above by $C \eps$, where $C$ is a constant independent of $\eps$.
\begin{rem}
We mention that, just like the deterministic case \cite{godu}, we are able to show that the rate of convergence for vanishing viscosity solution is $\frac12$. It is also worth mentioning that this rate is optimal.
\end{rem}
 
We finish this section by introducing a
  special class of entropy functions which  will play a crucial role in the analysis.  Let $\beta:\R \rightarrow \R$ be a $C^\infty$ function satisfying 
 \begin{align*}
      \beta(0) = 0,\quad \beta(-r)= \beta(r),\quad \beta^\prime(-r) = -\beta^\prime(r),\quad \beta^{\prime\prime} \ge 0,
 \end{align*} and 
\begin{align*}
\beta^\prime(r)=\begin{cases} -1,\quad &\text{when} ~ r\le -1,\\
                               \in [-1,1], \quad &\text{when}~ |r|<1,\\
                               +1, \quad &\text{when} ~ r\ge 1.
                 \end{cases}
\end{align*} For any $\xi > 0$, define  $\beta_\xi:\R \rightarrow \R$ by 
\begin{align*}
         \beta_\xi(r) = \xi\, \beta \left(\frac{r}{\xi} \right).
\end{align*} 
Then
\begin{align}\label{eq:approx to abosx}
 |r|-M_1\xi \le \beta_\xi(r) \le |r|\quad \text{and} \quad |\beta_\xi^{\prime\prime}(r)| \le \frac{M_2}{\xi} \mathds{1}_{ \lbrace |r|\le \xi \rbrace},
\end{align} 
where $\mathds{1}_{A}$ denotes the characteristic function of the set $A$, and
\begin{align*}
 M_1 = \sup_{|r|\le 1}\big | |r|-\beta(r)\big |, \quad M_2 = \sup_{|r|\le 1}|\beta^{\prime\prime} (r)|.
\end{align*}
Finally, by simply dropping $\xi$, for $\beta= \beta_\xi$ we define 
\begin{align*}
&F_k^\beta(a,b)=\int_{b}^a \beta^\prime(\sigma-b)F_k^\prime(\sigma)\,d(\sigma), \quad
F^\beta(a,b)=(F_1^\beta(a,b),F_2^\beta(a,b),...,F_d^\beta(a,b)),\\
&F_k(a,b)= \text{sign}(a-b)(F_k(a)-F_k(b)) , \quad
F(a,b)= (F_1(a,b),F_2(a,b),....,F_d(a,b)).
\end{align*} 


\section{A priori estimates}
\label{sec:apriori}
 In this section, we derive uniform spatial BV bound for the stochastic balance laws driven by 
 L\'{e}vy process given by  \eqref{eq:levy_stochconservation_laws} under the assumptions  \ref{A1}, \ref{A2}, \ref{A3}, and ~\ref{A4'}.

\begin{thm}[spatial bounded variation]
 \label{thm: spatial BV estimate}
   Let the assumptions ~\ref{A1}, ~\ref{A2}, ~\ref{A3}, and ~\ref{A4'} hold. Furthermore, let $u_\eps(t,x)$ be a
solution to the initial value problem \eqref{eq:levy_stochconservation_laws-viscous}. Then, for any time $t>0$
\begin{align*}
  E \Big[\int_{\R_x^d} \big| \grad u_\eps(t,x)\big|\,dx\Big]\le E \Big[\int_{\R_x^d} \big| \grad u_\eps(0,x)\big|\,dx\Big]\le 
  E \Big[\int_{\R_x^d} \big| \grad u_0(x)\big|\,dx\Big].
\end{align*}
\end{thm}

\begin{proof}
 Since $u_\eps(t,x)$ is a smooth solution of  the
initial value problem \eqref{eq:levy_stochconservation_laws-viscous}, by differentiating 
\eqref{eq:levy_stochconservation_laws-viscous} with respect to $x_i$, we find that $\partial_{x_i}u_\eps(t,x), 1\le i\le d$ satisfies
the stochastic partial differential equation given by
 \begin{align*}
 & d \big(\partial_{x_i}u_\eps(t,x)\big) + \mbox{div}_x \big( F^\prime_\eps(u_\eps(t,x))\partial_{x_i}u_\eps(t,x)\big) \,dt =  \int_{|z|> 0} \eta^\prime_\eps(u_\eps(t,x);z)\partial_{x_i}u_\eps(t,x)\tilde{N}(dz,\,dt) \notag \\
  & \hspace{10cm} + \epsilon \Delta_{xx} (\partial_{x_i}u_\eps(t,x)) \,dt.
\end{align*}
To proceed further, we apply It\^{o}-L\'{e}vy formula to $\beta_{\xi}(\partial_{x_i}u_\eps(t,x))$ to obtain
\begin{align}
 & d \big( \beta_{\xi}(\partial_{x_i}u_\eps(t,x))\big) + \mbox{div}_x \big( F^\prime_\eps(u_\eps(t,x))\partial_{x_i}u_\eps(t,x)\big)\,\beta_{\xi}^\prime(\partial_{x_i}u_\eps(t,x))
 \,dt  \notag \\
 &=  \int_{|z|> 0} \int_{\theta=0}^1 \eta^\prime_\eps(u_\eps(t,x);z)\partial_{x_i}u_\eps(t,x) \beta_{\xi}^\prime \big(\partial_{x_i}u_\eps(t,x)+\theta\,
 \eta^\prime_\eps(u_\eps(t,x);z)\partial_{x_i}u_\eps(t,x) \big) \,d\theta \,\tilde{N}(dz,\,dt) \notag \\
 & \quad + \int_{|z|> 0} \int_{\theta=0}^1 (1-\theta)\big(\eta^\prime_\eps(u_\eps;z)\partial_{x_i}u_\eps\big)^2 \beta_{\xi}^{\prime\prime} 
 \Big(\partial_{x_i}u_\eps(t,x)+\theta\,
 \eta^\prime_\eps(u_\eps(t,x);z)\partial_{x_i}u_\eps(t,x) \Big) \,d\theta\, \nu(dz)\,dt \notag \\
  & \hspace{7cm} + \epsilon \Delta_{xx} \big(\partial_{x_i}u_\eps(t,x)\big)\beta_{\xi}^\prime(\partial_{x_i}u_\eps(t,x)) \,dt.\label{eq:levy_stochconservation_laws-viscous-derivative-1}
\end{align}
Since $\beta_\xi$ is convex, we conclude that
\begin{align*}
\eps \Delta_{xx} \big(\partial_{x_i}u_\eps(t,x)\big)\beta_{\xi}^\prime\big(\partial_{x_i}u_\eps(t,x)\big)
 &= \eps \Big( \Delta \beta_{\xi}(\partial_{x_i}u_\eps(t,x))-\beta_{\xi}^{\prime\prime}(\partial_{x_i}u_\eps(t,x))|\grad \partial_{x_i}u_\eps(t,x)|^2\Big)\\
 &\, \le \eps  \Delta \beta_{\xi}\big(\partial_{x_i}u_\eps(t,x)\big),
\end{align*}
and for the martingale term, we have
\begin{align*}
E \Bigg[ \int_0^t  \int_{|z|> 0} \int_{\theta=0}^1 \eta^\prime_\eps(u_\eps(s,x);z)\partial_{x_i}u_\eps(s,x) \beta_{\xi}^\prime \big(\partial_{x_i}u_\eps(s,x)+\theta\,
 \eta^\prime_\eps(u_\eps(s,x);z)\partial_{x_i}u_\eps(s,x) \big) \,d\theta \,\tilde{N}(dz,\,ds) \Bigg] =0.
\end{align*}
 By Remark~\ref{rem:integrability}, we see that for each fixed $\eps>0$ and $\,1\le i\le d$, $\grad \partial_{x_i}u_\eps(t,x)$ is integrable.
 Let $0 \le \psi(x) \in C_c^\infty(\R^d)$. Multiply \eqref{eq:levy_stochconservation_laws-viscous-derivative-1} by 
 $\psi$ and  then integrate respect to $x$  to have
 \begin{align}
  & E \Big[\int_{\R_x^d}\beta_{\xi} \big(\partial_{x_i}u_\eps(t,x)\big)\psi(x)\,dx\Big]- 
  E \Big[\int_{\R_x^d}\beta_{\xi}\big(\partial_{x_i}u_\eps(0,x)\big) \psi(x)\,dx\Big]\notag \\
 & \qquad \le    E \Big[\int_{\R_x^d}\int_{s=0}^t \int_{|z|> 0} \int_{\theta=0}^1 (1-\theta)\beta_{\xi}^{\prime\prime} 
 \Big(\partial_{x_i}u_\eps(s,x)+\theta\,
 \eta^\prime_\eps(u_\eps(s,x);z)\partial_{x_i}u_\eps(s,x) \Big) \notag \\
 & \hspace{6cm}\times  \big(\eta^\prime_\eps(u_\eps(s,x);z)\partial_{x_i}u_\eps(s,x)\big)^2
 \psi(x)\,d\theta\, \nu(dz)\,ds\,dx \Big] \notag \\
 &  \hspace{2cm}- E \Big[\int_{\R_x^d}\int_{s=0}^t \mbox{div}_x \big( F^\prime_\eps(u_\eps(s,x))
 \partial_{x_i}u_\eps(s,x)\big)\beta_{\xi}^\prime\big(\partial_{x_i}u_\eps(s,x)\big)\psi(x)
 \,ds\,dx\Big]    \notag \\
  & \hspace{4cm}+ \eps E \Big[ \int_{\R_x^d}\int_{s=0}^t \beta_\xi \big(\partial_{x_i}u_\eps(s,x)\big) \Delta \psi(x)\,ds\,dx\Big].\label{eq:deriv}
 \end{align}
To proceed further, observe that
  \begin{align*}
   &\mbox{div}_x \big( F^\prime_\eps(u_\eps(s,x))\partial_{x_i}u_\eps(s,x)\big)\beta_{\xi}^\prime\big(\partial_{x_i}u_\eps(s,x)\big)\psi(x) 
   = \mbox{div}_x\Big[ F^\prime_\eps(u_\eps(s,x))\partial_{x_i}u_\eps(s,x)\beta_{\xi}^\prime(\partial_{x_i}u_\eps(s,x))\psi(x)\Big]\\
   & \qquad -  \partial_{x_i}u_\eps(s,x)\, F^\prime_\eps(u_\eps(s,x)) \Big( \beta_{\xi}^{\prime\prime}(\partial_{x_i}u_\eps(s,x))\,\psi(x)\,\grad \partial_{x_i}u_\eps(s,x) + \beta_{\xi}^{\prime}(\partial_{x_i}u_\eps(s,x))\,\grad \psi(x) \Big). 
  \end{align*} 
  Therefore, we obtain from \eqref{eq:deriv}
  \begin{align}
  & E \Big[\int_{\R_x^d}\beta_{\xi} \big(\partial_{x_i}u_\eps(t,x)\big)\psi(x)\,dx\Big] \le 
  E \Big[\int_{\R_x^d}\beta_{\xi} \big(\partial_{x_i}u_\eps(0,x)\big)  \psi(x)\,dx\Big]\notag \\
  & +  E \Big[\int_{\R_x^d}\int_{s=0}^t \int_{|z|> 0} \int_{\theta=0}^1  (1-\theta)\beta_{\xi}^{\prime\prime} 
 \Big(\partial_{x_i}u_\eps(s,x)+\theta\,
 \eta^\prime_\eps(u_\eps(s,x);z)\partial_{x_i}u_\eps(s,x) \Big) \notag \\
 & \hspace{7cm}\times \big(\eta^\prime_\eps(u_\eps(s,x);z)\partial_{x_i}u_\eps(s,x)\big)^2 \psi(x)
 \,d\theta\, \nu(dz)\,ds\,dx \Big] \notag \\
 & \qquad \quad+ E \Big[\int_{\R_x^d}\int_{s=0}^t \partial_{x_i}u_\eps(s,x) \psi(x)\beta_{\xi}^{\prime\prime}\big(\partial_{x_i}u_\eps(s,x)\big)\grad \partial_{x_i}u_\eps(s,x)
 \cdot F^\prime_\eps(u_\eps(s,x)) \,ds\,dx\Big] \notag\\
 & \qquad \qquad+ E \Big[\int_{\R_x^d}\int_{s=0}^t \partial_{x_i}u_\eps(s,x)\beta_{\xi}^{\prime}\big(\partial_{x_i}u_\eps(s,x)\big)\grad \psi(x)
 \cdot F^\prime_\eps(u_\eps(s,x)) \,ds\,dx\Big]  \notag\\
 & \qquad \quad \qquad+ \eps E \Big[\int_{\R_x^d}\int_{s=0}^t \beta_\xi \big(\partial_{x_i}u_\eps(s,x)\big)\Delta \psi(x)\,ds\,dx\Big] \notag \\
 & := E \Big[\int_{\R_x^d}\beta_{\xi}\big(\partial_{x_i}u_\eps(0,x)\big)\psi(x)\,dx\Big] + \mathcal{E}_1(\eps,\xi) +  \mathcal{E}_2(\eps,\xi)  +  \mathcal{E}_3(\eps,\xi)+  \mathcal{E}_4(\eps, \xi).\label{eq:levy_stoc-derivative-2}
  \end{align}
To estimate $\mathcal{E}_1(\eps,\xi)$, we proceed as follows. Note that we can rewrite $\mathcal{E}_1(\eps,\xi)$ as
\begin{align*}
\mathcal{E}_1(\eps,\xi)= E\Big[\int_{\R_x^d}\int_{s=0}^t \int_{|z|> 0} \int_{\theta=0}^1 (1-\theta)\,h^2 \beta_{\xi}^{\prime\prime} 
 \big(a+\theta\,h \big)\psi(x)\,d\theta\, \nu(dz)\,ds\,dx \Big],
   \end{align*} 
   where $a=\partial_{x_i}u_\eps(s,x)$ and $h= \eta^\prime_\eps(u_\eps(s,x);z)\partial_{x_i}u_\eps(s,x)$.
In view of the assumption ~\ref{A3}, it is easy to see that 
   \begin{align}
    h^2  \beta_{\xi}^{\prime\prime} (a+\theta\,h ) &
    \le \big|\partial_{x_i}u_\eps(s,x)\big|^2(1\wedge |z|^2)  \beta_{\xi}^{\prime\prime} (a+\theta\,h ).\label{estimate-h^2}
   \end{align}
Next we move on to find a suitable upper bound on $a^2  \beta_{\xi}^{\prime\prime} \big(a+\theta\,h \big)$. 
Since $\beta^{\prime\prime}$ is an even function, without loss of generality
   we may assume that $a>0$. Then by our assumption \ref{A3}
   \begin{align*}
    \partial_{x_i}u_\eps(t,x) + \theta \eta^\prime_\eps \big(u_\eps(t,x);z\big)\partial_{x_i}u_\eps(t,x) \ge (1-\lambda^*)
    \partial_{x_i}u_\eps(t,x),
   \end{align*} for $\theta \in [0,1]$. In other words
   \begin{align}
     0\le a \le (1-\lambda^*)^{-1} (a+ \theta\,h).\label{estimate-a}
   \end{align}
Combining \eqref{estimate-h^2} and \eqref{estimate-a} yields
   \begin{align*}
     h^2  \beta_{\xi}^{\prime\prime} (a+\theta\,h )   \le  (1\wedge |z|^2) (1-\lambda^*)^{-2} (a+\theta \,h)^2
     \beta_{\xi}^{\prime\prime} (a+\theta\,h )
      \le C (1\wedge |z|^2) \,\xi.
   \end{align*}
Since by assumption ~\ref{A4'}, $\int_{|z|>0}(1\wedge |z|^2) \,\nu(dz) < + \infty $, we infer that 
\begin{align}
\label{eq:1}
\abs{\mathcal{E}_1(\eps,\xi)} \le C\, t\,\xi\,\norm{\psi}_{L^1(\R^d)}\,\, \text{and hence}\,\, \mathcal{E}_1(\eps,\xi) \mapsto 0, \,\text{as}\,\, \xi \downarrow 0.
\end{align}
Next, we move on to estimate $ \mathcal{E}_2(\eps,\xi)$. In fact, we have
\begin{align*}
\abs{\mathcal{E}_2(\eps,\xi)} \le E \Big[ \int_{\R_x^d}\int_{s=0}^t |\partial_{x_i}u_\eps(s,x)|\psi(x)\beta_{\xi}^{\prime\prime} \big(\partial_{x_i}u_\eps(s,x)\big) 
    \big|\grad \partial_{x_i}u_\eps(s,x)\big||F^\prime_\eps(u_\eps(s,x))| \,ds\,dx \Big]
   \end{align*}
First observe that, in view of \eqref{eq:approx to abosx}, we obtain
  \begin{align*}
   |\partial_{x_i}u_\eps(s,x)|\beta_{\xi}^{\prime\prime} \big(\partial_{x_i}u_\eps(s,x)\big) &\le   |\partial_{x_i}u_\eps(s,x)| \frac{M_2}{\xi}
   \chi_{[-\xi, \xi]}(\partial_{x_i}u_\eps(s,x)) \mapsto 0, \,\, \text{almost surely as}\,\, \xi \downarrow 0,
   \end{align*}
  and moreover we see that
  \begin{align*}
   |\partial_{x_i}u_\eps(s,x)|\beta_{\xi}^{\prime\prime} \big(\partial_{x_i}u_\eps(s,x)\big)&\psi(x) \big|\grad \partial_{x_i}u_\eps(s,x)\big|\,|F^\prime_\eps(u_\eps(s,x))| \\
  &\le C ||\psi(\cdot)||_{L^\infty}\Big( |\grad \partial_{x_i}u_\eps(s,x)|^2 + |(u_\eps(s,x))|^{2p_0}\Big),\,\, \text{for some}~\,p_0 \in \mathbb{N}.
  \end{align*} 
In view of Remark \ref{rem:integrability} and Proposition \ref{prop:vanishing viscosity-solution}, the right-hand side
is integrable and independent of $\xi >0$. Therefore, one can apply dominated convergence theorem to conclude that
\begin{align}
\label{eq:2}
\mathcal{E}_2(\eps,\xi) \mapsto 0, \,\text{as}\,\, \xi \downarrow 0.
\end{align}
Next, we consider the term $\mathcal{E}_3(\eps,\xi)$. With the help of uniform estimates \eqref{uniform-estimates}, we conclude
\begin{align}
   |\mathcal{E}_3(\eps,\xi)| &\le E \Big[\int_{\R_x^d}\int_{s=0}^t  |\partial_{x_i}u_\eps(s,x)|\,|\grad \psi(x)| 
 | F^\prime_\eps(u_\eps(s,x))| \,ds\,dx\Big]\notag \\
 & \qquad \le ||\grad \psi(\cdot)||_{L^\infty(\R^d)} E \Big[\int_{\R_x^d}\int_{s=0}^t  \big|\partial_{x_i}u_\eps(s,x)\big|
 \big|u_\eps(s,x)\big|^p \,ds\,dx\Big] \notag \\
 & \qquad \qquad \le ||\grad \psi(\cdot)||_{L^\infty(\R^d)} E\Big[ \int_{\R_x^d}\int_{s=0}^t\Big(|\partial_{x_i}u_\eps(s,x)|^2 +
  \big|u_\eps(s,x)\big|^{2p} \Big)\,ds\,dx\Big] \notag \\
 & \qquad \qquad \qquad \le C(\eps)\, T\, ||\grad \psi(\cdot)||_{L^\infty(\R^d)}, \label{esti: C-derivative}
\end{align} 
where we have used that  for fixed $\eps>0$, $ \partial_{x_i}u_\eps(s,x)$  is integrable.

Finally we move on to estimate the term $\mathcal{E}_4(\eps,\xi)$. It is easy to see that 
\begin{align}
   |\mathcal{E}_4(\eps,\xi)| &\le ||\Delta \psi(\cdot)||_{L^\infty (\R_x^d)} \eps\, E \Big[\int_{\R_x^d} \int_{s=0}^t \big|\partial_{x_i}u_\eps(s,x)\big|\,ds\,dx\Big] \le T\,C(\eps) ||\Delta \psi(\cdot)||_{L^\infty (\R^d)}  
\label{esti: D-derivative}
\end{align}
Taking advantage of \eqref{eq:approx to abosx} in \eqref{eq:levy_stoc-derivative-2} helps us to conclude 
\begin{align}
    & E \Big[\int_{\R_x^d} \big|\partial_{x_i}u_\eps(t,x)\big|\psi(x)\,dx \Big] \le
  E \Big[\int_{\R_x^d} \big|\partial_{x_i}u_\eps(0,x)\big| \psi(x)\,dx\Big] 
   + M_1 \xi \,||\psi(\cdot)||_{L^1(\R^d)} \notag \\
   & \hspace{6cm} + \mathcal{E}_1(\eps,\xi) +  \mathcal{E}_2(\eps,\xi)  +  \mathcal{E}_3(\eps,\xi)+  \mathcal{E}_4(\eps, \xi).\label{eq: for Bv estimate}
   \end{align} 
In what follows, we combine all the above estimates \eqref{eq:1}, \eqref{eq:2}, \eqref{esti: C-derivative}, and \eqref{esti: D-derivative} and then send  $\xi \mapsto 0$ in \eqref{eq: for Bv estimate} 
to obtain
\begin{align}
 & E \Big[\int_{\R_x^d} \big|\partial_{x_i}u_\eps(t,x)\big|\psi(x)\,dx\Big]  \notag \\
  & \qquad \qquad \quad \le  E \Big[\int_{\R_x^d} \big|\partial_{x_i}u_\eps(0,x)\big|\psi(x)\,dx\Big]
 + C(\eps) \Big(||\Delta \psi(\cdot)||_{L^\infty(\R^d)} + ||\grad \psi(\cdot)||_{L^\infty(\R^d)}\Big)\,T\label{eq:final-Bv-estimate}
\end{align}
To this end, we define $0\le \psi_N(x)\in C_c^2(\R^d)$ such that
\begin{align*}
   \psi_N (x)=\begin{cases} 1\quad \text{when} ~ |x|\le N\\
                            0 \quad \text{when} ~ |x|> N+1.
\end{cases}    
\end{align*} 
Note that since \eqref{eq:final-Bv-estimate} holds for $\psi(x)=\psi_N(x)$, we choose $\psi(x)=\psi_N(x)$ 
in \eqref{eq:final-Bv-estimate}, and then sending $N \goto \infty$ to obtain
   \begin{align*}
E \Big[\int_{\R_x^d} \big|\partial_{x_i}u_\eps(t,x)\big|\,dx\Big] & \le
  E \Big[\int_{\R_x^d} \big|\partial_{x_i}u_\eps(0,x)\big|\,dx \Big],
   \end{align*} which completes the proof.
\end{proof}

An important and immediate corollary of the uniform spatial BV estimate is the existence of BV bounds for the entropy solution of \eqref{eq:levy_stochconservation_laws}. We have following theorem. 
\begin{thm} [BV entropy solution]
\label{thm:existence}
Suppose that the assumptions \ref{A2}, ~\ref{A3}, and ~\ref{A4'} hold. Then there exists an unique entropy solution of \eqref{eq:levy_stochconservation_laws} with initial data satisfying assumption ~\ref{A1} such that
\begin{align}
  \label{EQ:BV-ESTIMATE}    E \Big[\abs{u(t,\cdot)}_{BV(\R^d)} \Big] \le E \Big[\abs{u_0}_{BV(\R^d)} \Big], \,\, \text{for any $t>0$.}
\end{align}
\end{thm}
\begin{proof}
    We take advantage of the well-posedness results from \cite{kbm2014} and claim that the sequence $\{u_{\eps}(t,\cdot)\}$ converges, in the sense of Young measures, to the unique $L^p(\R^d)$-valued entropy solution $u(t,\cdot )$. In view of the uniform BV estimate in Theorem \ref{thm: spatial BV estimate}, by passing to the limit, we conclude \eqref{EQ:BV-ESTIMATE}. In other words, the unique $L^p$-valued entropy solution has bounded variation if the initial condition is $BV$.
\end{proof}

\section{Proof of The Main Theorem}
\label{cont-depen-estimate}
It is worth mentioning that, the average $L^1$-contraction principle [see, for example, \cite{kbm2014}] gives the continuous dependence on the initial data in stochastic balance laws
of the type \eqref{eq:levy_stochconservation_laws}. However, we intend to establish continuous dependence 
also on the nonlinearities, i.e., on the flux function and the noise coefficient. To achieve that, we need to consider the following regularized problem:
\begin{equation}
  \label{eq:levy_stochconservation_laws-viscous_2}
  \begin{cases}
  dv_\eps(s,y) + \mbox{div}_y G_\eps(v_\eps(s,y)) \,ds =  \int_{|z|> 0} \sigma_\eps(v_\eps(s,y);z)\tilde{N}(dz,\,ds)+ \eps \Delta_{yy} v_\eps (s,y)\,ds, & (s,y) \in \Pi_T, \\
  v_{\eps}(0,y) = v_0^{\eps}(y), \quad y \in \R^d;
\end{cases}
\end{equation} where $(v_0^\eps, \sigma_\epsilon, G_\epsilon)$ are regularized version of $(v_0, \sigma, G)$ satisfying the conditions in \eqref{eq:regularization}.
In view of Theorem \ref{thm:existence}, we conclude that $v_\eps(s,y)$ converges, as Young measures,  to the unique BV-entropy solution $v(s, y)$ of \eqref{eq:stability-12} with initial data $v_0(y)$. Let $u(t,\cdot)$ be the unique BV-entropy solution of  \eqref{eq:levy_stochconservation_laws} with initial data $u_0(x)$. Moreover, we assume that the assumptions \ref{A1}, \ref{A2}, \ref{A3}, and \ref{A4'} hold for both sets of given functions $(v_0, G, \sigma)$ and $(u_0, F, \eta)$.

We estimate the $L^1$-difference between two entropy solutions $u$ and $v$.  The theorem will be proved by using
the ``{\it doubling of variables}" technique. However, we can't directly compare two entropy solutions $u$ and $v$, but instead we first compare the entropy solution $u(t,x)$ with the solution of the viscous approximation \eqref{eq:levy_stochconservation_laws-viscous_2}, i.e., $v_{\eps}(s,y)$. This approach is somewhat different from the deterministic approach, where one can directly compare two entropy solutions. For
deterministic continuous dependence theory consult \cite{perthame,cockburn,chen2005,kal-resibro} and references therein.

To begin with, let $\rho$ and $\varrho$ be the standard mollifiers on $\R$ and  $\R^d$ respectively such that
  $\supp(\rho) \subset [-1,0)$ and $\supp(\varrho) = B_1(0)$. For $\delta > 0$ and $\delta_0 > 0$,
  let $\rho_{\delta_0}(r) = \frac{1}{\delta_0}\rho(\frac{r}{\delta_0})$ and
  $\varrho_{\delta}(x) = \frac{1}{\delta^d}\varrho(\frac{x}{\delta})$.
For a nonnegative test function $\psi\in C_c^{1,2}([0,\infty)\times \rd)$ with $|\grad \psi(t,x)| \le C\,\psi(t,x)$, $|\Delta \psi(t,x)| \le C\,\psi(t,x)$ and two positive constants $\delta, \delta_0 $, define            
 \begin{align}
\label{eq:doubled-variable} \phi_{\delta,\delta_0}(t,x, s,y) = \rho_{\delta_0}(t-s) \varrho_{\delta}(x-y) \psi(s,y). 
\end{align} 
Observe that $ \rho_{\delta_0}(t-s) \neq 0$ only if $s-\delta_0 \le t\le s$, and therefore $ \phi_{\delta,\delta_0}(t,x; s,y)= 0$
 outside  $s-\delta_0 \le t < s$.
 
 Furthermore, let $\varsigma$ be the standard symmetric 
nonnegative mollifier on $\R$ with support in $[-1,1]$ and $\varsigma_l(r)= \frac{1}{l} \varsigma(\frac{r}{l})$ 
for $l > 0$. We now write the entropy inequality for $u(t,x)$, based on the 
entropy pair $(\beta(\cdot-k), F^\beta(\cdot, k))$, and 
then multiply by $\varsigma_l(v_\eps(s,y)-k)$, integrate with 
respect to $ s, y, k$ and take the expectation. The result is
\begin{align}
0\le  & E \Big[\int_{\Pi_T}\int_{\R_x^d}\int_{\R_k} \beta(u(0,x)-k)
\phi_{\delta,\delta_0}(0,x,s,y) \varsigma_l(v_\eps(s,y)-k)\,dk \,dx\,dy\,ds\Big] \notag \\
 &\qquad + E \Big[\int_{\Pi_T} \int_{\Pi_T} \int_{\R_k} \beta(u(t,x)-k)\partial_t \phi_{\delta,\delta_0}(t,x,s,y)
\varsigma_l(v_\eps(s,y)-k)\,dk \,dx\,dt\,dy\,ds \Big]\notag \\ 
 & \qquad +  E \Big[ \int_{\Pi_T} \int_{\R_k}\int_{\Pi_T}\int_{|z|>0}\Big(\beta \big(u(t,x) +\eta(u(t,x);z)-k\big)-\beta(u(t,x)-k)\Big) \notag \\
& \hspace{6cm} \times \phi_{\delta,\delta_0}(t,x,s,y)\,\varsigma_l(v_\eps(s,y)-k) \,\tilde{N}(dz,dt) \,dx \,dk \,dy\,ds \Big] \notag\\
&\qquad +  E \Big[\int_{\Pi_T} \int_{t=0}^T\int_{|z|>0}\int_{\R_x^d} 
\int_{\R_k} \Big(\beta \big(u(t,x) +\eta(u(t,x);z)-k\big)-\beta(u(t,x)-k) \notag \\
 & \hspace{6.5cm}-\eta(u(t,x);z) \beta^{\prime}(u(t,x)-k)\Big)
 \phi_{\delta,\delta_0}(t,x;s,y) \notag \\
&\hspace{8cm}\times \varsigma_l(v_\eps(s,y)-k)\,dk\,dx\,\nu(dz)\,dt\,dy\,ds\Big]\notag \\
& \qquad +  E \Big[\int_{\Pi_T}\int_{\Pi_T} \int_{\R_k} 
 F^\beta(u(t,x),k) \cdot \grad_x \varrho_\delta(x-y)\,\psi(s,y)\,\rho_{\delta_0}(t-s)\notag \\
 &\hspace{9cm} \times \varsigma_l(v_\eps(s,y)-k)\,dk\,dx\,dt\,dy\,ds\Big] \notag \\
& =:  I_1 + I_2 + I_3 +I_4 + I_5. \label{stochas_entropy_1-levy}
\end{align}
 
 We now apply the It\^{o}-L\'{e}vy formula 
to \eqref{eq:levy_stochconservation_laws-viscous_2} and multiply with test function $\phi_{\delta_0, \delta}$ and $\varsigma_l(u(t,x)-k)$ and integrate . The result is
\begin{align}
 0\le  &\, E \Big[\int_{\Pi_T}\int_{\R_x^d}\int_{\R_k} 
 \beta(v_\eps(0,y)-k)\phi_{\delta,\delta_0}(t,x,0,y) \varsigma_l(u(t,x)-k)\,dk \,dx\,dy\,dt\Big] \notag \\
   & \qquad \qquad \qquad \quad +  E \Big[\int_{\Pi_T} \int_{\Pi_T} \int_{\R_k} 
 \beta(v_\eps(s,y)-k)\partial_s \phi_{\delta,\delta_0}(t,x,s,y)
 \varsigma_l(u(t,x)-k)\,dk \,dy\,ds\,dx\,dt\Big] \notag \\ 
  + &  E \Big[\int_{\Pi_T} \int_{\Pi_T}\int_{|z|>0} \int_{\R_k} 
 \Big(\beta \big(v_\eps(s,y) +\sigma_\eps(v_\eps(s,y);z)-k\big)
 -\beta(v_\eps(s,y)-k)\Big) \notag \\
 & \hspace{6.5cm} \times \phi_{\delta,\delta_0}(t,x,s,y)\varsigma_l(u(t,x)-k)\,dk \,\tilde{N}(dz,ds)\,dy\,dx\,dt \Big]\notag\\
  + &  E \Big[\int_{\Pi_T} \int_{s=0}^T\int_{|z|>0}\int_{\R_y^d} 
 \int_{\R_k} \Big(\beta \big(v_\eps(s,y) +\sigma_\eps(v_\eps(s,y);z)-k\big)
 -\beta(v_\eps(s,y)-k) \notag \\
  & \hspace{6.0cm}-\sigma_\eps(v_\eps(s,y);z) \beta^{\prime}(v_\eps(s,y)-k)\Big)
  \phi_{\delta,\delta_0}(t,x;s,y) \notag \\
 &\hspace{8.7cm}\times \varsigma_l(u(t,x)-k)\,dk\,dy\,\nu(dz)\,ds\,dx\,dt\Big]\notag \\
  &\quad+  E\Big[\int_{\Pi_T}\int_{\Pi_T} \int_{\R_k}  
 G_\eps^\beta(v_\eps(s,y),k)\cdot \grad_y\varrho_\delta(x-y) \psi(s,y)
  \rho_{\delta_0}(t-s ) \,\varsigma_l(u(t,x)-k)\,dk\,dx\,dt\,dy\,ds\Big] \notag \\
 &\quad + E \Big[\int_{\Pi_T}\int_{\Pi_T} \int_{\R_k}  
 G_\eps^\beta(v_\eps(s,y),k) \cdot \grad_y \psi(s,y) \varrho_\delta(x-y) 
  \rho_{\delta_0}(t-s ) \,\varsigma_l(u(t,x)-k)\,dk\,dx\,dt\,dy\,ds\Big] \notag \\
 &\quad - \eps  E \Big[\int_{\Pi_T} \int_{\Pi_T} \int_{\R_k} 
 \beta^\prime(v_\eps(s,y)-k)\grad_y v_\eps(s,y) \cdot \grad_y  \phi_{\delta,\delta_0}(t,x,s,y)
\,  \varsigma_l(u(t,x)-k)\,dk \,dy\,ds\,dx\,dt\Big],\label{stochas_entropy_2-levy}
\end{align} 
where $ G_\eps^\beta(a,b) = \int_a^b \beta^\prime(r-b)G^\prime_\eps(r)\,dr$. It follows by direct computations that there is $p\in \mathbb{N}$ such that
\[\big|G_\eps^\beta(a,b)- G^\beta(a, b) \big|\le C\eps \big(1+|a|^{2p}+|b|^{2p}\big). \] 
In view of the uniform moment estimates, it follows from 
\eqref{stochas_entropy_2-levy} that 
\begin{align}
 0\le  &\, E \Big[\int_{\Pi_T}\int_{\R_x^d}\int_{\R_k} 
 \beta(v_\eps(0,y)-k)\phi_{\delta,\delta_0}(t,x,0,y) \varsigma_l(u(t,x)-k)\,dk \,dx\,dy\,dt\Big] \notag \\
   & \qquad \qquad \qquad \quad +   E \Big[\int_{\Pi_T} \int_{\Pi_T} \int_{\R_k} 
 \beta(v_\eps(s,y)-k)\partial_s \phi_{\delta,\delta_0}(t,x,s,y)
 \varsigma_l(u(t,x)-k)\,dk \,dy\,ds\,dx\,dt\Big] \notag \\ 
  + &  E \Big[\int_{\Pi_T} \int_{\Pi_T}\int_{|z|>0} \int_{\R_k} 
 \Big(\beta \big(v_\eps(s,y) +\sigma_\eps(v_\eps(s,y);z)-k\big)
 -\beta(v_\eps(s,y)-k)\Big) \notag \\
 & \hspace{6.5cm} \times \phi_{\delta,\delta_0}(t,x,s,y)\varsigma_l(u(t,x)-k)\,dk \,\tilde{N}(dz,ds)\,dy\,dx\,dt \Big]\notag\\
  + &  E \Big[\int_{\Pi_T} \int_{s=0}^T\int_{|z|>0}\int_{\R_y^d} 
 \int_{\R_k} \Big(\beta \big(v_\eps(s,y) +\sigma_\eps(v_\eps(s,y);z)-k\big)
 -\beta(v_\eps(s,y)-k) \notag \\
  & \hspace{6.0cm}-\sigma_\eps(v_\eps(s,y);z) \beta^{\prime}(v_\eps(s,y)-k)\Big)
  \phi_{\delta,\delta_0}(t,x;s,y) \notag \\
 &\hspace{8.7cm}\times \varsigma_l(u(t,x)-k)\,dk\,dy\,\nu(dz)\,ds\,dx\,dt\Big]\notag \\
  & \quad +  E\Big[\int_{\Pi_T}\int_{\Pi_T} \int_{\R_k}  
 G^\beta(v_\eps(s,y),k)\cdot \grad_y\varrho_\delta(x-y) \psi(s,y)
  \rho_{\delta_0}(t-s ) \varsigma_l(u(t,x)-k)\,dk\,dx\,dt\,dy\,ds\Big] \notag \\
 &\quad + E \Big[\int_{\Pi_T}\int_{\Pi_T} \int_{\R_k}  
 G^\beta(v_\eps(s,y),k) \cdot \grad_y \psi(s,y) \varrho_\delta(x-y) 
  \rho_{\delta_0}(t-s ) \varsigma_l(u(t,x)-k)\,dk\,dx\,dt\,dy\,ds\Big] \notag \\
 -& \eps  E \Big[\int_{\Pi_T} \int_{\Pi_T} \int_{\R_k} 
 \beta^\prime(v_\eps(s,y)-k)\grad_y v_\eps(s,y) \cdot \grad_y  \phi_{\delta,\delta_0}
  \varsigma_l(u(t,x)-k)\,dk \,dy\,ds\,dx\,dt\Big] + C(\beta,\psi) \frac{\eps}{\delta} \notag \\
  & =: J_1 + J_2 + J_3 + J_4 + J_5 + J_6 + J_7 + C(\beta, \psi) \frac{\eps}{\delta}, \label{stochas_entropy_3-levy}
\end{align}
where $C(\beta, \psi)$ is a constant depending only on the quantities in the parentheses. Our aim is to add \eqref{stochas_entropy_1-levy} and \eqref{stochas_entropy_3-levy}, 
and pass to the  limits with respect to the various parameters involved. We do this by claiming
a series of lemma's and proofs of these lemmas follow from \cite{kbm2014} modulo cosmetic changes. 
 
 To begin with, note that particular choice of test function \eqref{eq:doubled-variable} implies that $J_1=0$. 
\begin{lem}
\label{stochastic_lemma_1}
It holds that 
\begin{align}
  I_1 + J_1   & \underset{\delta_0 \goto 0} \longrightarrow E \Big[\int_{\R_y^d}\int_{\R_x^d}\int_{\R_k} 
  \beta(u(0,x)-k)\psi(0,y)\varrho_{\delta} (x-y) \varsigma_l(v_\eps(0,y)-k)\,dk\,dx\,dy \Big]\notag\\
  &\underset{l \goto 0} \longrightarrow  E \Big[\int_{\R_y^d}\int_{\R_x^d}
  \beta(u(0,x)-v_\eps(0,y))\psi(0,y)\varrho_{\delta} (x-y)\,dx\,dy\Big]\notag.
 \end{align}
\end{lem}
We now turn our attention to $(I_2 + J_2)$. Since $\beta$, $\varsigma_l$ are even functions, we see that
 \begin{align*}
I_2+ J_2  =&E\Big[\int_{\Pi_T} \int_{ \Pi_T} \int_{\R_k}  
  \beta(v_\eps(s,y)-k) \partial_s\psi(s,y)\, \rho_{\delta_0}(t-s)
 \varrho_\delta(x-y)\\
 &\hspace{7.5cm}\times \varsigma_l(u(t,x)-k)\,dk\,dy\,ds\,dxdt\Big]. 
 \end{align*}
 
 \begin{lem}\label{stochastic_lemma_2}
It holds that
\begin{align*}
I_2 + J_2  &\underset{\delta_0 \goto 0}\longrightarrow   E\Big[ \int_{\Pi_T}\int_{\R_y^d}\int_{\R_k}
    \beta(v_\eps(s,y)-k) \partial_s\psi(s,y)\varrho_\delta(x-y)\varsigma_l(u(s,x)-k)\,dk\,dy\,dx\,ds\Big]\\
&\underset{l \goto 0}\longrightarrow  E \Big[\int_{\Pi_T}\int_{\R_y^d} \beta(v_\eps(s,y)-u(s,x)) \partial_s\psi(s,y)
 \, \varrho_\delta(x-y)\,dy\,dx\,ds\Big].
\end{align*}
\end{lem}
Next, we consider the term $I_5 + J_5$ and regarding these terms we have the following lemma.
\begin{lem}\label{stochastic_lemma_4}
The following hold:
\begin{align}
 \lim_{l\goto 0}\lim_{\delta_0 \goto 0} I_5
 = E\Big[\int_{s=0}^T \int_{\R_y^d}\int_{\R_x^d} F^\beta(u(s,x),v_\eps(s,y))\cdot \grad_x \varrho_\delta(x-y)
 \,\psi(s,y)\,dx\,dy\,ds\Big]\label{estim:I_5}
\end{align} 
and 
\begin{align}
 \lim_{l\goto 0}\lim_{\delta_0 \goto 0} J_5
 = E\Big[\int_{s=0}^T \int_{\R_y^d}\int_{\R_x^d} G^\beta(v_\eps(s,y),u(s,x))\cdot \grad_y \varrho_\delta(x-y)
 \,\psi(s,y)\,dx\,dy\,ds\Big]\label{estim:J_5}
\end{align} 
\end{lem} 
\begin{lem} \label{stochastic_lemma_5}
It holds that
\begin{align*}
J_6  &\underset{\delta_0 \goto 0}{\rightarrow}  E \Big[\int_{\Pi_T}\int_{\R_x^d}
\int_{\R_k}  G^\beta(v_\eps(s,y),k)\cdot\grad_y \psi(s,y)\,\varrho_\delta(x-y)  \varsigma_l(u(s,x)-k)\,dk\,dx\,dy\,ds\Big] \\
&\underset{l \goto 0}{\rightarrow}  
E \Big[\int_{\Pi_T}\int_{\R_x^d} G^\beta(v_\eps(s,y),u(s,x))\cdot \grad_y \psi(s,y) \varrho_\delta(x-y)\,dx\,dy\,ds\Big].
 \end{align*}
\end{lem}
Next, we consider the term $J_7$. Thanks to the uniform spatial $BV$ estimate for vanishing viscosity solution (cf. Theorem~\ref{thm: spatial BV estimate}), we conclude that
\begin{align}
|J_7| &\le \eps ||\beta^\prime||_{\infty} 
\Big|E \Big[\int_{\Pi_T}\int_{\R_x^d}|\grad_y v_\eps(s,y)| 
|\grad_y[\psi(s,y)\varrho_\delta(x-y)|\,dx\,dy\,ds\Big]\Big|\notag \\
&\le\eps \, ||\beta^\prime||_{\infty} E\Big[\int_{|y|\le K}\int_{t=0}^T\int_{\R_x^d}
|\grad_y v_\eps(t,y)|\,|\grad_y[\psi(t,y)\varrho_\delta(x-y)]|\,dx\,dt\,dy\Big]\notag \\
& \le C \frac{\eps}{\delta} E\big[ |v_0|_{BV(\R^d)}\big] \label{esti:J_7-eps}
\end{align}
\begin{lem}
\label{stochastic_lemma_6} 
It holds that
\begin{align}
&\lim_{l\goto 0}\lim_{\delta_{0}\goto 0} J_4 = E \Big[\int_{\Pi_T} 
\int_{\R_x^d}\int_{|z|>0}\int_{\lambda =0}^1 
(1-\lambda)\beta^{\prime\prime} \big(v_\eps(s,y)-u(s,x)
+\lambda \sigma_{\eps}(v_\eps(s,y);z)\big)\notag \\
&\hspace{4.5cm} \times|\sigma_{\eps}(v_\eps(s,y);z)|^2\psi(s,y)
\varrho_{\delta}(x-y)\,d\lambda \,\nu(dz)\,dx\,dy \,ds\Big],\label{eq:J_4-delta} \\
&\lim_{l\goto 0}\lim_{\delta_{0}\goto 0} I_4 
= E \Big[\int_{\Pi_T} \int_{\R_x^d}\int_{|z|>0} 
\int_{\lambda =0}^1  (1-\lambda)
\beta^{\prime\prime} \big(u(s,x)-v_\eps(s,y) 
+\lambda \eta(u(s,x);z)\big)\notag \\
&\hspace{4.5cm} \times|\eta(u(s,x);z)|^2
\psi(s,y)\varrho_{\delta}(x-y) \,d\lambda \,\nu(dz)\,dx\,dy \,ds\Big].\label{eq:I_4-delta}
\end{align}
\end{lem}
Finally, we consider the stochastic term 
$ I_3 + J_3$; 
\begin{lem}\label{stochastic_lemma_3} 
It holds that $J_3 = 0$ and 
\begin{align} 
\lim_{l\goto 0}\lim_{\delta_0 \goto 0} I_3 & = E\Big[\int_{\Pi_{T}}\int_{\R_x^d}\int_{|z|>0}
\Big( \beta(u(r,x)+ \eta(u(r,x);z)
-v_\eps(r,y)-\sigma_\eps(v_\eps(r,y);z)) \notag \\
& \hspace{3cm}-\beta(u(r,x)-v_\eps(r,y)-\sigma_\eps(v_\eps(r,y);z))+ \beta\big(u(r,x)-v_\eps(r,y)\big) \notag \\
&  \hspace{2.5cm}-\beta\big(u(r,x)+ \eta(u(r,x);z)-v_\eps(r,y)\big) \Big) 
 \psi(r,y)\,\varrho_{\delta}(x-y)\,\nu(dz)\,dx\,dy\,dr\Big]. \notag
\end{align}
\end{lem}
To proceed further, we combine Lemma~\ref{stochastic_lemma_3} and 
Lemma~\ref{stochastic_lemma_6} and conclude that
\begin{align}
& \lim_{l\goto 0}\lim_{\delta_0 \goto 0} 
\Big((I_3 +J_3)+ (I_4 + J_4)\Big) \notag \\
= &E \Big[\int_{\Pi_T}\int_{\R_x^d}\Big(\int_{|z|>0}\Big\{\beta \big(u(t,x)
-v_\eps(t,y)+\eta(u(t,x);z)-\sigma_\eps(v_\eps(t,y);z)\big)\notag \\
& \hspace{5cm} -\beta \big(u(t,x)-v_\eps(t,y)\big)-\big(\eta(u(t,x);z)-\sigma_\eps(v_\eps(t,y);z)\big) \notag \\ 
&\hspace{5.5cm} \times \beta^\prime \big(u(t,x)-v_\eps(t,y)\big)\Big\}\,\nu(dz)\Big) \psi(t,y)\varrho_\delta(x-y)\,dx\,dy\,dt\Big] \notag \\
=& E\Big[ \int_{r=0}^T\int_{|z| > 0} \int_{\R_y^d}\int_{ \R_x^d} \int_{\rho=0}^1
 \beta^{\prime \prime}\Big( u(r,x) -v_\eps(r,y) +\rho\big(\eta(u(r,x);z)-\sigma_\eps(v_\eps(r,y);z)\big)\Big)\notag\\
&\hspace{6.6cm}\times (1-\rho) \big| \eta(u(r,x);z)-\sigma_\eps(v_\eps(r,y);z)\big|^2  \psi(r,y) \notag \\
& \hspace{9cm} \times \varrho_\delta(x-y)\,d\rho\,dx \,dy\,\nu(dz) \,dr\Big]\label{eq:extra term}
\end{align}

We are now in a position to add \eqref{stochas_entropy_1-levy} and \eqref{stochas_entropy_3-levy}
and pass to the limits $\underset{l\rightarrow 0}\lim\, \underset{\delta_0\downarrow 0} \lim$. In what follows, 
invoking Lemma~\ref{stochastic_lemma_1}, Lemma~\ref{stochastic_lemma_2}, Lemma~\ref{stochastic_lemma_4}, and Lemma \ref{stochastic_lemma_5}, and the expressions \eqref{esti:J_7-eps} and \eqref{eq:extra term}, 
we arrive at
\begin{align}
 0\le &  E \Big[\int_{\R_y^d}\int_{\R_x^d}
  \beta(u(0,x)-v_\eps(0,y))\psi(0,y)\varrho_{\delta} (x-y)\,dx\,dy\Big] \notag \\
&\hspace{5cm} +   E \Big[\int_{\Pi_T}\int_{\R_y^d} \beta(v_\eps(s,y)-u(s,x)) \partial_s\psi(s,y)\varrho_\delta(x-y)\,dy\,dx\,ds\Big]\notag \\
   & \qquad \, - E \Big[\int_{\Pi_T}\int_{\R_y^d} \nabla_y \cdot \{ G^\beta\big(v_\eps(s,y), u(s,x))-F^\beta\big(u(s,x),v_\eps(s,y)\big)\}
  \psi(s,y)\varrho_\delta(x-y)\,dy\,dx\,ds\Big] \notag \\
  & \quad +  E \Big[\int_{\Pi_T}\int_{\R_y^d} F^\beta\big(u(s,x),v_\eps(s,y)\big)\cdot  \grad_y\psi(s,y)\,\varrho_\delta(x-y)\,dy\,dx\,ds\Big] + C \Big(E\big[ |v_0|_{BV(\R^d)}\big] +1 \Big)\frac{\eps}{\delta} \notag \\
  + &  E\Big[ \int_{r=0}^T\int_{|z| > 0} \int_{\R_y^d}\int_{ \R_x^d} \int_{\rho=0}^1
 \beta^{\prime \prime}\Big( u(r,x) -v_\eps(r,y) +\rho\big(\eta(u(r,x);z)-\sigma_\eps(v_\eps(r,y);z)\big)\Big)\notag\\
&\hspace{6.6cm}\times (1-\rho) \big| \eta(u(r,x);z)-\sigma_\eps(v_\eps(r,y);z)\big|^2  \psi(r,y) \notag \\
& \hspace{9.6cm} \times \varrho_\delta(x-y)\,d\rho\,dx \,dy\,\nu(dz) \,dr\Big]\notag \\
&:=\mathcal{A}_1 + \mathcal{A}_2 + \mathcal{A}_3 +\mathcal{A}_4 +\mathcal{A}_5 + C \Big(E\big[ |v_0|_{BV(\R^d)}\big] +1 \Big)\frac{\eps}{\delta} .\label{estimate:A_0}
\end{align}
Again, our aim is to estimate all the above terms suitably. First observe that, since $\beta_\xi(r)\le |r|$, we obtain
 \begin{align}
  |\mathcal{A}_1| \le  E \Big[\int_{\R_y^d}\int_{ \R_x^d}\big| v_\eps(0,y) -u(0,x)\big| \psi(0,y)\,\varrho_\delta(x-y)
  \,dx\, dy\Big]. \label{estimate:A_11}
 \end{align}
 Next, by our choice of $\beta=\beta_\xi$, we have 
 \begin{align}
  \Big| \frac{\partial}{\partial v}\Big(F^{\beta_\xi}(u,v)-F^{\beta_\xi}(v,u)\Big)\Big|
  &=\Big|- F^{\prime}(v)\beta_\xi^{\prime}(v-u) - F^{\prime}(v)\beta_\xi^{\prime}(0) + 
  \int_{s=u}^v \beta_\xi^{\prime\prime}(s-v) F^\prime(s)\,ds\Big| \notag  \\
  &=\Big| \big(F^{\prime}(v)-F^{\prime}(u) \big)\beta_\xi^{\prime}(u-v) -
  \int_{s=u}^v \beta_\xi^{\prime}(s-v) F^{\prime\prime}(s)\,ds\Big| \notag \\
  & = \Big| \int_{u}^v \Big(\beta_\xi^{\prime}(u-v) -\beta_\xi^{\prime}(s-v) \Big) F^{\prime\prime}(s)\,ds \Big|
   \le M_2\,\xi\, ||F^{\prime\prime}||_{\infty}.\label{estimate:derivative-entropy-flux}
 \end{align} 
 Also from the definition of $F^{\beta}$ and $G^{\beta}$, it is evident that
 \begin{align}
  \Big| \frac{\partial}{\partial v}\Big(F^\beta(v,u)-G^\beta(v,u)\Big)\Big|\le |F^\prime(v)-G^\prime(v)|
  \label{estimate:derivative-different-entropy-flux}
 \end{align} Therefore, by \eqref{estimate:derivative-entropy-flux} and \eqref{estimate:derivative-different-entropy-flux}, we obtain
 \begin{align}
  \Big| \frac{\partial}{\partial v}\Big(F^\beta(u,v)-G^\beta(v,u)\Big)\Big|\le  M_2\,\xi\, ||F^{\prime\prime}||_{\infty}
  + |F^\prime(v)-G^\prime(v)| \label{estimate:derivative-different-entropy-flux-1}
 \end{align}
Keeping in mind the estimate \eqref{estimate:derivative-different-entropy-flux-1}, we proceed further by rewriting the term $\mathcal{A}_3$ as 
\begin{align*}
 \mathcal{A}_3= E \Big[\int_{\Pi_T}\int_{\R_y^d} \nabla_y v_\eps(s,y) \cdot \partial_v \big(
 F^\beta(u,v)-G^\beta(v,u)\big)\Big|_{(u,v)=(u(s,x),v_\eps(s,y))}
  \psi(s,y)\varrho_\delta(x-y)\,dy\,dx\,ds\Big] 
\end{align*}
Thanks to the uniform spatial $BV$ estimate for vanishing viscosity solution (cf. Theorem~\ref{thm: spatial BV estimate}), we conclude that
 \begin{align}
  |\mathcal{A}_3|  & \le \Big( M_2\,\xi\, ||F^{\prime\prime}||_{\infty}
  + ||F^\prime-G^\prime||_{\infty}\Big) E \Big[\int_{s=0}^T \int_{\R_y^d}\int_{\R_x^d} |\grad_y v_\eps(s,y)|
  \psi(s,y) \varrho_\delta(x-y)\,dx\,dy\,ds\Big] \notag \\
  &\le  E\Big[|v_0|_{BV(\R^d)}\Big] \Big( M_2\,\xi\, ||F^{\prime\prime}||_{\infty}
  + ||F^\prime-G^\prime||_{\infty}\Big) \int_{s=0}^T ||\psi(s,\cdot)||_{L^\infty(\R^d)}\,ds.\label{estimate:A_31}
 \end{align}
 Next, we recall that the function $\psi(t,x)$ satisfies $|\grad \psi(t,x)| \le C\,\psi(t,x)$ and $|F^{\beta}(a,b)| \le ||F^\prime||_{\infty} |a-b|$ for any $a,b \in \R$.
Therefore, we conclude
 \begin{align}
  |\mathcal{A}_4| \le & C ||F^\prime||_{L^\infty} E \Big[\int_{s=0}^T \int_{\R_y^d}\int_{\R_x^d} \big|u(s,x)-v_\eps(s,y)\big|
  \psi(s,y) \varrho_\delta(x-y)\,dx\,dy\,ds\Big] \notag \\
  \le & C ||F^\prime||_{L^\infty} E \Big[\int_{s=0}^T \int_{\R_y^d}\int_{\R_x^d}\beta_{\xi} \big(u(s,x)-v_\eps(s,y)\big)
  \psi(s,y) \varrho_\delta(x-y)\,dx\,dy\,ds\Big] \notag \\
  & \qquad \qquad \qquad \qquad \qquad+ C M_1\,||F^\prime||_{L^\infty}\,\xi \int_{s=0}^T ||\psi(s,\cdot)||_{L^{\infty}(\R^d)}\,ds.\label{estimate:A_41}
 \end{align} 
Let us focus on  the term $\mathcal{A}_5$. For this, let us define
 \begin{align*}
  a:=u(r,x)-v_\eps(r,y),\quad  \text{and} \quad
 b:=\eta(u(r,x);z)-\sigma_\eps(v_\eps(r,y);z).
 \end{align*} 
 Then $\mathcal{A}_5$ can be rewritten in the following simplified form
  \begin{align}
   \mathcal{A}_5&= E\Big[ \int_{r=0}^T\int_{|z| > 0} \int_{\R_y^d}\int_{ \R_x^d} \int_{\rho=0}^1
 (1-\rho) b^2  \beta^{\prime \prime}\big(a +\rho\,b\big) \psi(r,y)\,\varrho_\delta(x-y)\,d\rho\,dx \,dy\,\nu(dz) \,dr\Big] \notag \\
 & \le C  E\Big[ \int_{r=0}^T\int_{|z| > 0} \int_{\R_y^d}\int_{ \R_x^d} \int_{\rho=0}^1 \big| \eta(u(r,x);z)-\sigma(u(r,x);z)\big|^2
 \beta^{\prime \prime}\big(a +\rho\,b\big) \notag \\
 & \hspace{8cm} \times \psi(r,y)\,\varrho_\delta(x-y)\,d\rho\,dx \,dy\,\nu(dz) \,dr\Big]  \notag \\
 & \qquad  + C  E\Big[ \int_{r=0}^T\int_{|z| > 0} \int_{\R_y^d}\int_{ \R_x^d} \int_{\rho=0}^1 \big| \sigma(u(r,x);z)-\sigma (v_\eps(r,y);z)\big|^2
 \beta^{\prime \prime}\big(a +\rho\,b\big)  \notag \\
 & \hspace{8cm} \times \psi(r,y)\,\varrho_\delta(x-y)\,d\rho\,dx \,dy\,\nu(dz) \,dr\Big]  \notag \\
 & \qquad  + C  E\Big[ \int_{r=0}^T\int_{|z| > 0} \int_{\R_y^d}\int_{ \R_x^d} \int_{\rho=0}^1 \big| \sigma (v_\eps(r,y);z)-\sigma_{\eps}(v_\eps(r,y);z)\big|^2
 \beta^{\prime \prime}\big(a +\rho\,b\big)  \notag \\
 & \hspace{8cm} \times \psi(r,y)\,\varrho_\delta(x-y)\,d\rho\,dx \,dy\,\nu(dz) \,dr\Big]  \notag \\
 & := \mathcal{A}_5^1 + \mathcal{A}_5^2  + \mathcal{A}_5^3 .
 \label{estimate:A_5^0}
  \end{align}

 To this end we recall that $\mathcal{D}(\eta, \sigma) =\displaystyle{\sup_{u \in \R}}\int_{|z|> 0} \frac{|\eta(u,z)-\sigma(u,z)|^2}{1+|u|^2}\nu(\,dz) $, which is well-defined in view of \ref{A3}.  With this quantity at hand it is easy see that
 
 \begin{align}
  \mathcal{A}_5^1 &  \le  \frac{C \mathcal{D}(\eta, \sigma)}{\xi} E \Big[\int_{r=0}^T\int_{ \R_x^d}\int_{ \R_y^d}
 (1+|u(r,x)|^2) \psi(r,y)\rho_\delta(x-y)\,dy \,dx \,dr\Big] \notag\\
&  \le   \frac{C \mathcal{D}(\eta, \sigma)}{\xi} \Big(\int_0^T ||\psi(s, \cdot)||_{L^1}\, ds +\int_0^T ||\psi(r,\cdot)||_{\infty} \,dr\Big)\label{estimate:A_5^1}
\end{align}

Next, we move on to estimate the term $\mathcal{A}_5^2$. Observe that
\begin{align}
 \big|\sigma(u(r,x);z)-\sigma (v_\eps(r,y);z)\big|^2 \beta^{\prime \prime}(a+\rho\,b)
 &\le  \big|u(r,x)-v_\eps(r,y)\big|^2 (1\wedge |z|^2) \beta^{\prime \prime}(a+\rho\,b) \notag \\
 &=  (1\wedge |z|^2) \, a^2\,\beta^{\prime \prime}(a+\rho\,b). \label{eq:nonlocal}
\end{align} 
Therefore, it is required to find a suitable upper bound on $a^2\,\beta^{\prime \prime}(a+\rho\,b)$.
Since $\beta^{\prime \prime}$ is non-negative and symmetric around zero, without loss of generality, we may assume that
$ a >0$.  Then, by our assumption \ref{A3}, we conclude that
\begin{align*}
  \big| &\eta(u(r,x);z)  -\sigma_\eps(v_\eps(r,y);z)\big| \\
 & \le   \big|\eta(u(r,x);z)-\sigma(u(r,x);z)\big| + \big|\sigma(u(r,x);z)-\sigma(v_\eps(r,y);z)\big| + \big|\sigma(v_\eps(r,y);z)-\sigma_{\eps}(v_\eps(r,y);z)\big|\notag \\
   & \le  \big|\eta(u(r,x);z)-\sigma(u(r,x);z)\big| + \lambda^* a + C \eps (1 + \abs{v_{\eps}}),\notag 
 \end{align*}which implies that
 \begin{align}
  a + \rho\,b \ge   - \big|\eta(u(r,x);z)-\sigma(u(r,x);z)\big| - C \eps (1 + \abs{v_{\eps}}) + (1- \lambda^*) a,\notag 
 \end{align} for $\rho \in[0,1]$. In other words
 \begin{align}
  0\le a \le (1-\lambda^*)^{-1} \Big\{ a +\rho\,b + \big|\eta(u(r,x);z)-\sigma(u(r,x);z)\big|  + C \eps (1 + \abs{v_{\eps}})\Big\}. \label{eq:nonlocal-1}
 \end{align}
 Now, we shall make use of \eqref{eq:nonlocal-1} in \eqref{eq:nonlocal}, to obtain
 \begin{align*}
  &\big|\sigma(u(r,x);z)-\sigma(v_\eps(r,y);z)\big|^2 \beta_\xi^{\prime \prime}(a+\rho\,b) \notag \\
  & \qquad \le (1-\lambda^*)^{-2} \Big\{ (a +\rho\,b)^2  + C\,\big|\eta(u(r,x);z)-\sigma(u(r,x);z)\big|^2 + C \eps^2 \left(1 + \abs{v_{\eps}}^2 \right) \Big\} 
  (1\wedge |z|^2)  \beta_\xi^{\prime \prime}(a+\rho\,b) \notag \\
  & \qquad  \le   C\Big( \xi +\frac{ \big|\eta(u(r,x);z)-\sigma(u(r,x);z)\big|^2}{\xi} + \frac{\eps^2 \left(1 + \abs{v_{\eps}}^2 \right)}{\xi}\Big) (1\wedge |z|^2).
 \end{align*}
 This helps us to conclude
\begin{align}
 \abs{\mathcal{A}_5^2}&
  \le  C E\Bigg[ \int_{r}\int_{|z| > 0} \int_{\R_y^d}\int_{ \R_x^d} 
\Big( \xi  + \frac{\eps^2 \left(1 + \abs{v_{\eps}}^2 \right)}{\xi}\Big)(1\wedge |z|^2)
  \psi(r,y)\,\varrho_\delta(x-y)\,dx \,dy\,m(dz) \,dr\Bigg]\notag \\
  &\qquad+ \frac{\mathcal{D}(\eta, \sigma)}{ \xi} \int_0^T\int_{\R_x^d}\int_{\R_y^d}(1+|u(r,x)|^2)\psi(r,y)\rho_{\delta}(x-y)\,dx\,dy\,dr\notag\\
& \le C (\xi +\frac{\eps^2}{\xi} ) \int_{s=0}^T || \psi(s,\cdot)||_{L^\infty(\R^d)}\,ds + 
  \frac{C \mathcal{D}(\eta, \sigma)}{\xi} \Big(\int_0^T ||\psi(s, \cdot)||_{L^1}\, ds + \int_0^T  ||\psi(r,\cdot)||_{\infty}\,dr\Big) .\label{estimate:A_5^2}
\end{align}
Next, we move on to estimate the term $\mathcal{A}_5^3$. In fact, it follows that 
\begin{align}
 &E\Big[ \int_{r=0}^T\int_{|z| > 0} \int_{\R_y^d}\int_{ \R_x^d} \int_{\rho=0}^1 \big| \sigma (v_\eps(r,y);z)-\sigma_{\eps}(v_\eps(r,y);z)\big|^2
 \beta^{\prime \prime}\big(a +\rho\,b\big)  \notag \\
 & \hspace{8cm} \times \psi(r,y)\,\varrho_\delta(x-y)\,d\rho\,dx \,dy\,\nu(dz) \,dr\Big] \notag\\
 & \qquad \quad \le C E\Big[ \int_{r=0}^T \int_{|z| > 0} \int_{\R_y^d}\int_{ \R_x^d} 
 \frac{\eps^2 \left(1 + \abs{v_{\eps}}^2 \right)}{\xi}\, (1\wedge |z|^2)
  \psi(r,y)\,\varrho_\delta(x-y)\,dx \,dy\,\nu(dz) \,dr\Big]\notag \\
&\qquad \quad  \le C \frac{\eps^2}{\xi}  \int_{s=0}^T || \psi(s,\cdot)||_{L^\infty(\R^d)}\,ds.
 \label{estimate:A_5^3}
\end{align}

We now make use of the estimates \eqref{estimate:A_5^1}, \eqref{estimate:A_5^2} and \eqref{estimate:A_5^3}. Then it is evident from \eqref{estimate:A_5^0} that
\begin{align}
 |\mathcal{A}_5| & \le   \frac{C \mathcal{D}(\eta, \sigma)}{\xi} \Big(\int_0^T ||\psi(s, \cdot)||_{L^1}\, ds +\int_0^T  ||\psi()||_{\infty} \,dr\Big)  \notag \\
   & \hspace{5cm} + C (\xi + \frac{\eps^2}{\xi} )\int_{s=0}^T || \psi(s,\cdot)||_{L^\infty(\R^d)}\,ds. \label{estimate:A_5}
\end{align}
Finally, we make use of the estimates \eqref{estimate:A_11}, \eqref{estimate:A_31}, \eqref{estimate:A_41} and \eqref{estimate:A_5} in \eqref{estimate:A_0} and 
pass to the limit as $\eps \goto 0 $ (keeping $\delta$ and $\xi$ fixed) in the resulting expression to conclude that
\begin{align}
0\le  & E \Big[\int_{\R_y^d}\int_{ \R_x^d}\big| v_0(y) -u(0,x)\big| \psi(0,y)\,\varrho_\delta(x-y)\,dx\, dy\Big] \notag \\
& +  E\Big[|v_0|_{BV(\R^d)}\Big] \Big( M_2\,\xi\, ||F^{\prime\prime}||_{\infty}
  + ||F^\prime-G^\prime||_{\infty}\Big) \int_{s=0}^T ||\psi(s,\cdot)||_{L^\infty(\R^d)}\,ds \notag \\
& + C ||F^\prime||_{L^\infty} E \Big[\int_{s=0}^T \int_{\R_y^d}\int_{\R_x^d}\beta_{\xi} \big(u(s,x)-v(s,y)\big)
  \psi(s,y) \varrho_\delta(x-y)\,dx\,dy\,ds\Big] \notag \\
&  \quad+ C \big(M_1\,||F^\prime||_{L^\infty} +1 \big)\,\xi \int_{s=0}^T ||\psi(s,\cdot)||_{L^{\infty}(\R^d)}\,ds \notag \\
&  \quad +\frac{C \mathcal{D}(\eta, \sigma)}{\xi} \Big(\int_0^T ||\psi(s, \cdot)||_{L^1}\, ds +\int_0^T  ||\psi(r,\cdot)||_{\infty} \,dr\Big)  \notag \\
& \qquad \quad +   E \Big[\int_{\Pi_T}\int_{\R_y^d} \beta(v(s,y)-u(s,x)) \partial_s\psi(s,y)\varrho_\delta(x-y)\,dy\,dx\,ds\Big].\label{eq:stoc-1}
\end{align}
 Now we can safely pass the limit as $\delta \goto 0$ in \eqref{eq:stoc-1} to obtain
\begin{align}
0\le  & E \Big[\int_{ \R_x^d}\big| v_0(x) -u(0,x)\big| \psi(0,x)\,dx\Big] \notag \\
& +  E\Big[|u_0|_{BV(\R^d)}\Big] \Big( M_2\,\xi\, ||F^{\prime\prime}||_{\infty}
  + ||F^\prime-G^\prime||_{\infty}\Big) \int_{s=0}^T ||\psi(s,\cdot)||_{L^\infty(\R^d)}\,ds \notag \\
& \quad + C ||F^\prime||_{L^\infty} E \Big[\int_{s=0}^T \int_{\R_x^d}\beta_{\xi} \big(v(s,x)-u(s,x)\big)
  \psi(s,x)\,dx\,ds\Big] \notag \\
&  \qquad + C \big(M_1\,||F^\prime||_{L^\infty} +1 \big)\,\xi \int_{s=0}^T ||\psi(s,\cdot)||_{L^{\infty}(\R^d)}\,ds \notag \\
&  \quad\quad +\frac{C \mathcal{D}(\eta, \sigma)}{\xi} \Big(\int_0^T ||\psi(s, \cdot)||_{L^1}\, ds + \int_0^T  ||\psi(r,\cdot)||_{\infty} \,dr\Big)  \notag \\
& \qquad \qquad +   E \Big[\int_{\Pi_T} \beta_{\xi}(u(s,x)-v(s,x)) \partial_s\psi(s,x)\,dx\,ds\Big].\label{eq:stoc-2}
\end{align}
 To proceed further, we make a special choice for the function $\psi(t,x)$. To this end, for each $h>0$ and fixed $t\ge 0$, we define
  \begin{align}
\psi_h^t(s)=\begin{cases} 1, &\quad \text{if}~ s\le t, \notag \\
                          1-\frac{s-t}{h}, &\quad \text{if}~~t\le s\le t+h,\notag \\
                          0, & \quad \text{if} ~ s \ge t+h.
            \end{cases}
\end{align}
Furthermore, let $\phi \in C_c^2(\R^d)$ be a cut-off function such that $|\grad \phi(x)| \le C \phi(x),~~ |\Delta \phi(x)|\le C \phi(x)$.
Clearly, \eqref{eq:stoc-2} holds with $\psi(s,x)=\psi_h^t(s)\phi(x)$. Let $\mathbb{T}$ be the set all points $t$ in
$[0, \infty)$ such that $t$ is  right Lebesgue point of 
$$A(s)= E\Big[\int_{\R_x^d} \beta_\xi\big(v(s,x)-u(s,x)\big)\,\phi(x)\,dx\Big].$$
Clearly, $\mathbb{T}^{\complement}$(complement of $\mathbb{T}$) has zero Lebesgue measure. Fix  $t\in \mathbb{T}$.
Then from \eqref{eq:stoc-2}, keeping in mind that we used generic $\beta$ for the function $\beta_{\xi}$, we obtain
\begin{align}
0\le  & E \Big[\int_{ \R_x^d}\big| v_0(x) -u(0,x)\big| \psi(0,x)\,dx\Big] \notag \\
& +  E\Big[|v_0|_{BV(\R^d)}\Big] \Big( M_2\,\xi\, ||F^{\prime\prime}||_{\infty}
  + ||F^\prime-G^\prime||_{\infty}\Big) ||\phi(\cdot)||_{L^\infty(\R^d)}\int_{s=0}^T\psi_h^t(s) \,ds \notag \\
& \quad + C ||F^\prime||_{L^\infty} E \Big[\int_{s=0}^T \int_{\R_x^d}\beta_{\xi} \big(v(s,x)-u(s,x)\big)
  \psi_h^t(s)\phi(x)\,dx\,ds\Big] \notag \\
&  \qquad + C \big(M_1\,||F^\prime||_{L^\infty} +1 \big)\,\xi  ||\phi(\cdot)||_{L^{\infty}(\R^d)}\int_{s=0}^T \psi_h^t(s)\,ds \notag \\
&  \quad\quad +\frac{C \mathcal{D}(\eta, \sigma)}{\xi} \Big(\int_0^T \int_{\R^d}\phi(x)\psi_h^t(s)\,dx\, ds +  \int_0^T\psi_h^t(s) ||\phi||_\infty \,dr\Big)  \notag \\
&  \qquad \qquad -\frac{1}{h}  \int_{s=t}^{t+h}  E\Big[ \int_{\R_x^d} \beta_\xi \big(u(s,x)-v(s,x)\big)\phi(x)\,dx \Big]\,ds .\label{eq:stoc-3}
\end{align} 
Since $t$ is a right Lebesgue point of $A(s)$, letting $h\goto 0$ in \eqref{eq:stoc-3} yields
\begin{align}
 E\Big[ \int_{\R_x^d} &\beta_\xi \big(u(t,x)-v(t,x)\big)\phi(x)\,dx \Big]\notag \\
\le  & E \Big[\int_{ \R_x^d}\big| v_0(x) -u(0,x)\big| \phi(x)\,dx\Big] 
+ C \big(M_1\,||F^\prime||_{L^\infty} +1 \big)\,\xi  ||\phi(\cdot)||_{L^{\infty}(\R^d)}\,t \notag \\
& +  E\Big[|v_0|_{BV(\R^d)}\Big] \Big( M_2\,\xi\, ||F^{\prime\prime}||_{\infty}
  + ||F^\prime-G^\prime||_{\infty}\Big) ||\phi(\cdot)||_{L^\infty(\R^d)}\,t  \notag \\
& \quad + C ||F^\prime||_{L^\infty}  \int_{s=0}^t E \Big[ \int_{\R_x^d}\beta_{\xi} \big(v(s,x)-u(s,x)\big)
\phi(x)\,dx\Big]\,ds \notag \\ 
&  \qquad  +\frac{Ct \mathcal{D}(\eta, \sigma)}{\xi} \Big(||\phi||_{L^1}+||\phi||_{L^\infty}\Big)  \notag
\end{align} for almost every $t> 0$. An weaker version of Grownwall's inequality then yields
\begin{align}
 E\Big[ \int_{\R_x^d} \beta_\xi \big(u(t,x)-v(t,x)\big) &\phi(x)\,dx \Big] 
 \le e^{C \,t\,||F^\prime||_{\infty}} E \Big[\int_{ \R_x^d}\big| v_0(x) -u(0,x)\big| \phi(x)\,dx\Big] \notag \\
 & + C  e^{C ||F^\prime||_{\infty}\,t} \Bigg\{ \big(M_1\,||F^\prime||_{L^\infty} +1 \big)\,\xi  ||\phi(\cdot)||_{L^{\infty}(\R^d)}\,t \notag \\
& +  E\Big[|v_0|_{BV(\R^d)}\Big] \Big( M_2\,\xi\, ||F^{\prime\prime}||_{\infty}
  + ||F^\prime-G^\prime||_{\infty}\Big) ||\phi(\cdot)||_{L^\infty(\R^d)}\,t  \notag \\
&  +\frac{Ct \mathcal{D}(\eta, \sigma)}{\xi} \Big(||\phi||_{L^1}+||\phi||_{L^\infty}\Big)  \Bigg\}\label{eq:stoc-4}
\end{align} for almost every $t >0$. Next, we recall that $|r| \le \beta_{\xi}(r) + M_1\, \xi$, for any $r \in \R$. Using this inequality, \eqref{eq:stoc-4} reduces to 
 \begin{align}
 E\Big[ \int_{\R_x^d} \beta_\xi \big(u(t,x)-v(t,x)\big) &\phi(x)\,dx \Big] 
 \le e^{C \,t\,||F^\prime||_{\infty}} E \Big[\int_{ \R_x^d}\big| v_0(x) -u(0,x)\big| \phi(x)\,dx\Big] + M_1\,\xi ||\phi(\cdot)||_{L^1(\R^d)} \notag \\
 & + C  e^{C ||F^\prime||_{\infty}\,t} \Bigg\{ \big(M_1\,||F^\prime||_{L^\infty} +1 \big)\,\xi  ||\phi(\cdot)||_{L^{\infty}(\R^d)}\,t \notag \\
& +  E\Big[|v_0|_{BV(\R^d)}\Big] \Big( M_2\,\xi\, ||F^{\prime\prime}||_{\infty}
  + ||F^\prime-G^\prime||_{\infty}\Big) ||\phi(\cdot)||_{L^\infty(\R^d)}\,t  \notag \\
&  +\frac{Ct \mathcal{D}(\eta, \sigma)}{\xi} \Big(||\phi||_{L^1}+||\phi||_{L^\infty}\Big)  \Bigg\}\label{eq:stoc-5}
\end{align} We now simply choose $\xi = \sqrt{t \mathcal{D}(\eta, \sigma)}$ and conclude that for a.e $t > 0$
\begin{align}
   E \Big[\int_{\R_x^d}    \big|u(t,x)-v(t,x)\big|\phi(x)dx \Big] \le&  C_T \,E \Big[\int_{\R_x^d}| u_0(x) -v_0(x)| \phi(x)\,dx\Big]+ E \big[|v_0|_{BV(\R^d)}\big]\, ||F^\prime-G^\prime||_{\infty}\, t\, ||\phi(\cdot)||_{L^{\infty}(\R^d)} \Big] \notag \\
  & +  C_T \,\Big[ \big( 1+ E[|v_0|_{BV(\R^d)}]\big) \sqrt{t \mathcal{D}(\eta, \sigma)} ||\phi(\cdot)||_{L^{\infty}(\R^d)} 
  + \sqrt{t \mathcal{D}(\eta, \sigma)}
 ||\phi(\cdot)||_{L^1(\R^d)} \Big], \notag
 \end{align} for some nonnegative constant $C_T$, independent of $|u_0|_{BV(\R^d)}$ and
$|v_0|_{BV(\R^d)}$. This completes the first part of the proof, and second part
follows from this by exploiting the specific structure of the test function $\phi(x)$.

\section{Proof of The Main Corollary}
\label{sec:error_estimate}
It is already known that the vanishing viscosity solutions converge (in an appropriate sense) to the unique 
entropy solution of the stochastic conservation law. However, the nature of such convergence described by a rate of convergence is not available. As a by product of the Main Theorem, we explicitly obtain the rate of convergence of vanishing viscosity solutions to the unique $BV$-entropy solution of the underlying problem \eqref{eq:levy_stochconservation_laws}. 
 
 By similar arguments as in the proof of the Main Theorem (cf. Section~\ref{cont-depen-estimate}), we arrive at
 \begin{align}
0\le  & E \Big[\int_{\R_y^d}\int_{ \R_x^d}\big| u_\eps(0,y) -u_0(x)\big| \psi(0,y)\,\varrho_\delta(x-y)\,dx\, dy\Big] \notag \\
& +  E\Big[|u_0|_{BV(\R^d)}\Big] M_2\,\xi\, ||F^{\prime\prime}||_{\infty} \int_{s=0}^T ||\psi(s,\cdot)||_{L^\infty(\R^d)}\,ds + C \frac{\eps^2}{\xi} \int_{s=0}^T ||\psi(s,\cdot)||_{L^\infty(\R^d)}\,ds  \notag \\
& \quad + C ||F^\prime||_{L^\infty} E \Big[\int_{s=0}^T \int_{\R_y^d}\int_{\R_x^d}\beta_{\xi} \big(u_\eps(s,y)-u(s,x)\big)
  \psi(s,y) \varrho_\delta(x-y)\,dx\,dy\,ds\Big] \notag \\
&  \qquad + C \big(M_1\,||F^\prime||_{L^\infty} +1 \big)\,\xi \int_{s=0}^T ||\psi(s,\cdot)||_{L^{\infty}(\R^d)}\,ds
+ C \Big( 1+ E\big[ |u_0|_{BV(\R^d)}\big] \Big)\frac{\eps}{\delta}\notag \\
& \qquad \quad + E \Big[\int_{\Pi_T}\int_{\R_y^d} \beta_\xi (u_\eps(s,y)-u(s,x)) \partial_s\psi(s,y)\varrho_\delta(x-y)\,dy\,dx\,ds\Big].\label{eq:stoc-6}
\end{align}

Let $\psi(s,y)=\psi_h^t(s)\phi(y)$ where $\psi_h^t(s)$ and $\phi(x)$ are  described previously.
Let $\mathbb{T}$ be the set all points $t$ in $[0, \infty)$ such that $t$ is  right Lebesgue point of 
$$B(s)= E \Big[\int_{\R_y^d}\int_{\R_x^d} \beta_\xi\big(u_\eps(s,y)-u(s,x)\big)\phi(y) \varrho_\delta(x-y)\,dx\,dy\Big].$$
Clearly, $\mathbb{T}^{\complement}$ has zero Lebesgue measure. Fix  $t\in \mathbb{T}$.
Thus, from  \eqref{eq:stoc-6}, we have
 \begin{align}
  & \frac{1}{h}\int_{s=t}^{t+h}  E \Big[\int_{\R_y^d}\int_{\R_x^d} \beta_\xi\big( u_\eps(s,y) -u(s,x)\big)\phi(y)
\varrho_\delta(x-y)\,dx\,dy \Big]\,ds \notag \\
  & \le   C||F^\prime||_{L^\infty} \int_{s=0}^{t+h}   E \Big[\int_{\R_y^d}\int_{\R_x^d}  \phi(y) \beta_\xi\big( u_\eps(s,y)
  -u(s,y)\big) \varrho_\delta(x-y)\psi_h^t(s)\,dx\,dy \Big]\,ds  \notag \\
 & \qquad  +  E \Big[\int_{\R_y^d}\int_{ \R_x^d}\big| u_\eps(0,y) -u_0(x)\big| \phi(y)\,\varrho_\delta(x-y) \,dx\, dy\Big]\notag \\
  & \hspace{1.0cm} + C\,E\big[|u_0|_{BV(\R^d)}\big] M_2\,\xi\, ||F^{\prime\prime}||_{\infty}
   ||\phi(\cdot)||_{L^\infty(\R^d)} \int_{s=0}^T \psi_h^t(s) \,ds + C \frac{\eps^2}{\xi} ||\phi(\cdot)||_{L^\infty(\R^d)}  \int_{s=0}^T \psi_h^t(s) \,ds \notag \\
  & \hspace{2cm}  + C \xi\, ||\phi(\cdot)||_{L^{\infty}(\R^d)} \int_{s=0}^T \psi_h^t(s)\,ds +
  C \Big( 1 + E\big[ |u_0|_{BV(\R^d)}\big]\Big) \frac{\eps}{\delta}. \notag
 \end{align} 
 Taking limit as $h\goto 0$, we have
 \begin{align}
  &  E \Big[\int_{\R_y^d}\int_{\R_x^d} \beta_\xi\big( u_\eps(t,y) -u(t,x)\big)\phi(y)\varrho_\delta(x-y)\,dx\,dy \Big] \notag \\
  & \quad \le   C||F^\prime||_{L^\infty} \int_{s=0}^{t}   E \Big[\int_{\R_y^d}\int_{\R_x^d}  \phi(y) \beta_\xi\big( u_\eps(s,y)
  -u(s,y)\big) \varrho_\delta(x-y)\,dx\,dy \Big]\,ds  \notag \\
 & \qquad + E \Big[\int_{\R_y^d}\int_{ \R_x^d}\big| u_\eps(0,y) -u_0(x)\big| \phi(y)\,\varrho_\delta(x-y) \,dx\, dy\Big] 
  + C \Big( 1+ E\big[ |u_0|_{BV(\R^d)}\big]\Big) \frac{\eps}{\delta} \notag \\
  & \qquad \quad + C \Big( 1 + E\big[|u_0|_{BV(\R^d)}\big]\Big) \xi\,||\phi(\cdot)||_{L^\infty(\R^d)} \,t +C \frac{\eps^2}{\xi} ||\phi(\cdot)||_{L^\infty(\R^d)}\, t\, \notag
 \end{align} 
 By an weaker version of Gronwall's inequality, for a.e $t > 0$
 \begin{align}
  &  E \Big[\int_{\R_y^d}\int_{\R_x^d} \beta_\xi\big( u_\eps(t,y) -u(t,x)\big)\phi(y)\varrho_\delta(x-y)\,dx\,dy \Big] \notag \\
  & \le   e^{C||F^\prime||_{L^\infty}\,t} \Big\{  E \Big[\int_{\R_y^d}\int_{ \R_x^d}\big| u_\eps(0,y) -u_0(x)\big|
  \phi(y)\,\varrho_\delta(x-y) \,dx\, dy\Big] + C \Big( 1+ E\big[ |u_0|_{BV(\R^d)}\big]\Big) \frac{\eps}{\delta} \Big\} \notag \\
  & \qquad + C e^{C||F^\prime||_{L^\infty}\,t} \Big[\Big( 1 + E\big[|u_0|_{BV(\R^d)}\big]\Big) \xi\,||\phi(\cdot)||_{L^\infty(\R^d)} \,t  + \frac{\eps^2}{\xi} ||\phi(\cdot)||_{L^\infty(\R^d)}\, t \Big] \notag
 \end{align} 
  Since $|r| \le M_1 \xi + \beta_{\xi}(r)$, we have
   \begin{align}
  &  E \Big[\int_{\R_y^d}\int_{\R_x^d} \big | u_\eps(t,y) -u(t,x)\big|\phi(y)\varrho_\delta(x-y)\,dx\,dy \Big] \notag \\
  & \le   e^{C||F^\prime||_{L^\infty}\,t} \Big\{  E \Big[\int_{\R_y^d}\int_{ \R_x^d}\big| u_\eps(0,y) -u_0(x)\big|
  \phi(y)\,\varrho_\delta(x-y) \,dx\, dy\Big] + C \Big( 1+ E\big[ |u_0|_{BV(\R^d)}\big]\Big) \frac{\eps}{\delta} \Big\} \notag \\
  & \qquad + C e^{C||F^\prime||_{L^\infty}\,t} \Big[\Big( 1 + E\big[|u_0|_{BV(\R^d)}\big]\Big) \xi\,||\phi(\cdot)||_{L^\infty(\R^d)} \,t  + \frac{\eps^2}{\xi} ||\phi(\cdot)||_{L^\infty(\R^d)}\, t \Big]
  + C \xi\, ||\phi(\cdot)||_{L^\infty(\R^d)}.\label{eq:stoc-7}
 \end{align} 
 First we send  $\phi$  to $\chi_{\R^d}$ in \eqref{eq:stoc-7}, and then choose $\xi= \eps$. The resulting estimate  gives
   \begin{align}
  &  E \Big[\int_{\R_y^d}\int_{\R_x^d} \big | u_\eps(t,y) -u(t,x)\big|\varrho_\delta(x-y)\,dx\,dy \Big] \notag \\
  & \le   e^{C||F^\prime||_{L^\infty}\,t} \Big\{  E \Big[\int_{\R_y^d}\int_{ \R_x^d}\big| u_\eps(0,y) -u_0(x)\big|
  \,\varrho_\delta(x-y) \,dx\, dy\Big] + C \Big( 1+ E\big[ |u_0|_{BV(\R^d)}\big]\Big) \frac{\eps}{\delta} \Big\} \notag \\
  & \qquad + C e^{C||F^\prime||_{L^\infty}\,t} \Big( 1 + E\big[|u_0|_{BV(\R^d)}\big]\Big) \eps \,t 
  + C \eps.\label{eq:stoc-8}
 \end{align} 
Notice that,
 \begin{align}
  &E \Big[\int_{\R_y^d}  \big|u_\eps(t,y)-u(t,y)\big|\,dy\Big]  \notag \\
  \le &   E \Big[\int_{\R_y^d}\int_{\R_x^d}  \big| u_\eps(t,y) -u(t,x)\big|
\varrho_\delta(x-y)\,dx\,dy \Big]
 + E \Big[\int_{\R_y^d}\int_{\R_x^d}   \big| u(t,x) -u(t,y)\big|\varrho_\delta(x-y)\,dx\,dy \Big]\notag \\
\le & E \Big[ \int_{\R_y^d}\int_{\R_x^d} \big| u_\eps(t,y) -u(t,x)\big|\varrho_\delta(x-y)\,dx\,dy \Big]
  + \delta\, E\Big[|u_0|_{BV(\R^d)}\Big], \label{estimate:solu}
 \end{align} 
and 
 \begin{align}
   E \Big[\int_{\R_y^d}\int_{ \R_x^d} \big| u_\eps(0,y) -u_0(x)\big|\varrho_\delta(x-y) \,dx\, dy\Big]
   \le  E \Big[\int_{\R_x^d} \big| u_\eps(0,x) -u_0(x)\big| \,dx \Big]+ \delta\, E\Big[|u_0|_{BV(\R^d)}\Big] \label{estimate:ini}
 \end{align} 
 We combine  \eqref{estimate:solu} and \eqref{estimate:ini} in \eqref{eq:stoc-8}
  to conclude
   \begin{align}
  &  E \Big[\int_{\R_y^d} \big | u_\eps(t,y) -u(t,y)\big|\,dy \Big] \notag \\
  & \le   e^{C||F^\prime||_{L^\infty}\,t} \Big\{  E \Big[\int_{\R_y^d}\big| u_\eps(0,y) -u_0(y)\big| dy\Big]
  + C \Big( 1+ E\big[ |u_0|_{BV(\R^d)}\big]\Big) \frac{\eps}{\delta} + \delta\, E\big[|u_0|_{BV(\R^d)}\big] \Big\} \notag \\
  & \qquad + C e^{C||F^\prime||_{L^\infty}\,t} \Big( 1 + E\big[|u_0|_{BV(\R^d)}\big]\Big) \eps \,t 
  + C \eps + \delta\, E\big[|u_0|_{BV(\R^d)}\big].\label{eq:stoc-9}
 \end{align} 
We choose $\delta= \eps^{\frac{1}{2}}$ in \eqref{eq:stoc-9}, and conclude that, for a.e $t> 0$,
  \begin{align}
    E \Big[\int_{\R_x^d} \big|u_\eps(t,x) & -u(t,x)\big|\,dx\Big] \notag \\
   \le & C(T) \Big \{ \eps^{\frac{1}{2}} \big( 1 + E[|u_0|_{BV(\R^d)}] \big)(1+ t)
     + E \Big[\int_{\R_x^d} \big|u_\eps(0,x)-u_0(x)\big|\,dx\Big] \Big\}, \notag
  \end{align} 
for some constant $C(T)>0$, independent of $ E\big[|u_0|_{BV(\R^d)}\big]$. This completes the proof.

\section{Fractional BV Estimates}
\label{sec:frac}
In this section, we consider a more general class of stochastic balance laws driven by L\'{evy} noise of the type
 \begin{equation}
 \label{eq:levy_stochconservation_laws_spatial}
 \begin{cases} 
 du(t,x) + \mbox{div}_x F(u(t,x))\,dt
 =\int_{|z|> 0} \eta( x,u(t,x);z) \, \tilde{N}(dz,dt), & \quad x \in \Pi_T, \\
 u(0,x) = u_0(x), & \quad  x\in \R^d,
\end{cases}
\end{equation}
Observe that, the noise coefficient $\eta(x,u;z)$ depends explicitly on the spatial position $x$. 
Moreover, we assume that
$\eta(x,u;z)$ satisfies the following assumptions:
\begin{Assumptions2} 
 \item \label{B3}  There exist  positive constants $K>0$  and  $\lambda^* \in (0,1)$  such that 
 \begin{align*}
 | \eta(x,u;z)-\eta(y,v;z)|  \leq  (\lambda^* |u-v| + K|x-y|)( |z|\wedge 1), ~\text{for all}~  u,v \in \R;~~z\in \R 
 ;~~ x,y \in \R^d.
 \end{align*}
\item \label{B4} There exists  a  non-negative function  $ g(x)\in L^\infty(\R^d)\cap L^2(\R^d)$  such that 
\begin{align*}
|\eta(x,u;z)| \le g(x)(1+|u|)(|z|\wedge 1), ~\text{for all}~ (x,u,z)\in \R^d \times  \R\times \R.  
\end{align*} 
\end{Assumptions2}
Clearly, our continuous dependence estimate is not applicable for problems of type  \eqref{eq:levy_stochconservation_laws_spatial}, and primary reason for that lies in the nonavailability  of $BV$ estimate here.  We refer to \cite[Section $2$]{Chen-karlsen2012} for discussion on this point for diffusion driven balance laws. However, it is possible to obtain a fractional $BV$ estimate. 
To that context, drawing primary motivation from the discussions in \cite{Chen-karlsen2012}, we intend to show that a uniform fractional $BV$ estimate can be obtained for the solution of the regularized stochastic parabolic problem given by
\begin{align} du_\eps(t,x) + \mbox{div}_x F_\eps(u_\eps(t,x))\,dt =  
  \int_{|z|> 0} \eta_\eps(x,u_\eps(t,x);z)\tilde{N}(dz, dt)+ \eps \Delta_{xx} u_\eps(t,x) \,dt,\label{eq:stability-3}
  \end{align} 
  where $F_\eps,~~\eta_\eps$ satisfy \eqref{eq:regularization}. Regarding equation \eqref{eq:stability-3}, we mention that existence and regularity 
  of the solution to the problem \eqref{eq:stability-3} has been studied in \cite{kbm2014}.
We start with a deterministic lemma, related to the estimation of the modulus of continuity
of a given integrable function, and also an useful link between Sobolev and Besov spaces. In fact, we have the following lemma, a proof of which can be found in \cite[Lemma $2$]{Chen-karlsen2012}.
  \begin{lem}\label{lem: deterministic-modulus-continuity}
    Let $h:\R^d\goto \R$ be a given integrable function, $ 0\le \phi \in C_c^\infty(\R^d)$ and \{$J_\delta\}_{\delta>0}$ be
     a sequence of symmetric mollifiers, i.e., $J_\delta(x)=\frac{1}{\delta^d} J(\frac{|x|}{\delta}),\, 0\le J \in C_c^\infty(\R)$,
    $\mbox{supp}(J) \subset [-1,1],\, J(-\cdot)= J(\cdot)$ and $\int J=1$. Then
    \begin{itemize}
     \item [(a)] For $r,s \in (0,1)$ with $r<s$, there exists a finite
    constant $C_1=C_1(J,d,r,s)$ such that
    \begin{align}
    \int_{\R_z^d}\int_{\R_x^d} | h(x+z)-h(x-z)| & J_\delta(z)\phi(x)\,dx\,dz \notag \\
    \le & C_1\,\delta^r \sup_{|z|\le \delta} |z|^{-s}\int_{\R_x^d} |h(x+z)-h(x-z)|\phi(x)\,dx.\label{lem:modulus-continuity-part-1}
    \end{align}
    \item[(b)] For $r,s \in (0,1)$ with $r<s$, there exists a finite
    constant $C_2=C_2(J,d,r,s)$ such that
    \begin{align}
   \sup_{|z|\le \delta}\int_{\R_x^d} & |h(x+z)-h(x)|\phi(x)\,dx \notag \\
     & \quad \le C_2 \delta^r \sup_{0<\delta\le 1} \delta^{-s}  \int_{\R_z^d}\int_{\R_x^d} | h(x+z)-h(x-z)| J_\delta(z)\phi(x)\,dx\,dz 
     + C_2 \delta^r ||h||_{L^1(\R^d)}.\label{lem:modulus-continuity-part-2}
    \end{align} 
 \end{itemize}
  \end{lem}
  
  Now we are in a position to state and prove a theorem regarding fractional BV estimation of solutions of \eqref{eq:stability-3}.
\begin{thm}[Fractional BV estimate]
  \label{thm:fractional BV estimate}
 Let the assumptions ~\ref{A1}, ~\ref{A2}, ~\ref{B3}, ~\ref{B4}, and ~\ref{A4'} hold. Let $u_\eps$ be a solution of \eqref{eq:stability-3} with the initial data $u_0(x)$ belongs to the Besov space $B^{\mu}_{1, \infty} (\R^d)$ for some $\mu \in (\frac12,1)$. 
Moreover, we assume that $F_\eps^{\prime\prime} \in L^\infty$. Then, for fixed $T>0$ and $R>0$, there exits a constant $C(T,R)$, 
independent of $\eps$, such that for any $0< t <T$,
 \begin{align*}
  \sup_{|y| \le \delta} E\Big[\int_{x\in K_R} \big|u_\eps(t,x+y)-u_\eps(t,x)\big|\,dx\Big] \le C(T,R)\,\delta^r,
 \end{align*} 
for some $r \in (0,\frac{1}{2})$ and $K_{R}:=\{x: |x| \le R\}$.
\end{thm}
  
 \begin{proof}
Let $0\le \phi(x)\in C_c^2(\R^d)$ be any test function such that $|\grad \phi(x)| \le C \phi(x)$ and $|\Delta \phi(x)|\le C \phi(x)$ for some constant
$C>0$. Let $J_\delta$ be a sequence of mollifier in $\R^d$ as mentioned in Lemma \ref{lem: deterministic-modulus-continuity}.
Consider the test function 
$$\psi_\delta(x,y): =J_\delta \left(\frac{x-y}{2} \right)\, \phi \left(\frac{x+y}{2} \right).$$ 
Sutracting two solutions $u_\eps(t,x)$, $u_\eps(t,y)$ of \eqref{eq:stability-3}, 
and applying It\^{o}-L\'{e}vy formula to that resulting equations, we obtain
\begin{align}
 & \beta_\xi\big( u_\eps(t,x) -u_\eps(t,y)\big)-
 \beta_\xi\big( u_\eps(0,x) -u_\eps(0,y)\big) \notag \\
 =& \int_{s=0}^t \beta_\xi^\prime\big(u_\eps(s,x)-u_\eps(s,y)\big)\Big(\mbox{div}_y F_\eps(u_\eps(s,y))- \mbox{div}_x
 F_\eps(u_\eps(s,x))\Big)\,ds \notag \\ 
   & \hspace{4cm}+ \eps\, \int_{r=0}^t \beta_\xi^\prime \big( u_\eps(r,x) -u_\eps(r,y)\big) \Big(\Delta_{xx}u_\eps(r,x)-\Delta_{yy} u_\eps(r,y)\Big)\, dr \notag \\
 & + \int_{r=0}^t\int_{|z| > 0} \int_{\rho=0}^1 (1-\rho)
 \beta_\xi^{\prime \prime}\Big( u_\eps(r,x) -u_\eps(r,y)+\rho\big(\eta_\eps(x,u_\eps(r,x);z)-\eta_\eps(y,u_\eps(r,y);z)\big)\Big)\notag
\\&\hspace{6cm}\times\big| \eta_\eps(x,u_\eps(r,x);z)
-\eta_\eps(y,u_\eps(r,y);z)\big|^2\,d\rho\,\nu(dz) \,dr\notag \\
& +  \int_{r=0}^t\int_{|z|> 0} 
 \Big[\beta_\xi \big( u_\eps(r,x) -u_\eps(r,y)+\eta_\eps(x,u_\eps(r,x);z)-\eta_\eps(y,u_\eps(r,y);z)\big) \notag \\
 &\hspace{8.5cm} - \beta_\xi \big( u_\eps(r,x) -u_\eps(r,y)\big)\Big]\, \tilde{N}(dz,dr). \notag 
\end{align}  To this end, we see that
\begin{align}
 \beta_\xi^\prime(u-v)\big(\Delta_{xx} u- \Delta_{yy} v\big) = \Big(\Delta_{xx}+ 2 \grad_x \cdot \grad_y + \Delta_{yy}\Big) \beta_\xi(u-v) 
 -\beta_\xi^{\prime\prime}(u-v)|\grad_x u-\grad_y v|^2. \label{expresion for beta-1}
\end{align}
Moreover, a simple calculation reveals that
\begin{align*}
 \Big(\Delta_{xx}+ 2 \grad_x \cdot \grad_y + \Delta_{yy}\Big) \psi_\delta(x,y)= \Delta \phi(\frac{x+y}{2})J_\delta(\frac{x-y}{2}), \\
\big(\grad_x + \grad_y\big) \psi_\delta(x,y)=\grad \phi(\frac{x+y}{2}) J_\delta(\frac{x-y}{2}).  
\end{align*}
Using convexity of $\beta_\xi$ and  \eqref{expresion for beta-1}, we have
\begin{align}
 &\int_{\R_y^d}\int_{\R_x^d} \beta_\xi \big( u_\eps(t,x) -u_\eps(t,y)\big) \psi_\delta(x,y) \,dx \,dy
- \int_{\R_y^d}\int_{\R_x^d} \beta_\xi\big( u_\eps(0,x) -u_\eps(0,y)\big) \psi_\delta(x,y) \,dx\, dy \notag \\
& \le \int_{s=0}^t \int_{\R_y^d}\int_{\R_x^d} F_\eps^\beta \big(u_\eps(s,x),u_\eps(s,y)\big)\cdot \grad \phi(\frac{x+y}{2}) J_\delta(\frac{x-y}{2})\,dx\,dy\,ds \notag \\
  & + \int_{s=0}^t \int_{\R_y^d}\int_{\R_x^d}\Big(F_\eps^\beta\big(u_\eps(s,y),u_\eps(s,x)\big)- F_\eps^\beta \big(u_\eps(s,x),u_\eps(s,y)\big)\Big)\cdot \grad_y \psi_\delta(x,y)\,dx
  \,dy\,ds \notag \\
  &+  \int_{r=0}^t \int_{\R_y^d}\int_{\R_x^d} 
  \eps\,\beta_\xi\big( u_\eps(r,x) -u_\eps(r,y)\big) J_\delta(\frac{x-y}{2}) \Delta \phi(\frac{x+y}{2})\, dx\, dy\, dr \notag \\
 & + \int_{r=0}^t\int_{|z| > 0} \int_{\R_y^d}\int_{\R_x^d} \int_{\rho=0}^1
 \beta_\xi^{\prime \prime}\Big( u_\eps(r,x) -u_\eps(r,y)+\rho\big(\eta_\eps(x,u_\eps(r,x);z)-\eta_\eps(y,u_\eps(r,y);z)\big)\Big)\notag \\
&\hspace{3cm}\times\big| \eta_\eps(x,u_\eps(r,x);z)
-\eta_\eps(y,u_\eps(r,y);z)\big|^2  \psi_\delta(x,y)\,d\rho\,dx \,dy\,\nu(dz) \,dr\notag \\
& +  \int_{r=0}^t\int_{|z|> 0} \int_{\R_y^d}\int_{\R_x^d} 
 \Big[\beta_\xi \big( u_\eps(r,x) -u_\eps(r,y)+\eta_\eps(x,u_\eps(r,x);z)-\eta_\eps(y,u_\eps(r,y);z)\big) \notag \\
 &\hspace{4cm} - \beta_\xi \big( u_\eps(r,x) -u_\eps(r,y)\big)\Big] 
\psi_\delta(x,y) \,dx \,dy \, \tilde{N}(dz,dr). \notag 
\end{align} 
Notice that since $\Big| F_\eps^\beta(u,v)-F_\eps^\beta(v,u)\Big| \le C ||F^\prime||_{\infty}\,\xi |u-v|$, we obtain
\begin{align}
& E \Big[\int_{\R_y^d}\int_{\R_x^d} \beta_\xi\big( u_\eps(t,x) -u_\eps(t,y)\big) \psi_\delta(x,y) \,dx \,dy\Big]
- E \Big[\int_{\R_y^d}\int_{\R_x^d} \beta_{\xi}\big( u_\eps(0,x) -u_\eps(0,y)\big) \psi_\delta(x,y) \,dx\, dy\Big] \notag \\
& \le C ||F^\prime||_{\infty}\int_{s=0}^t\,E \Big[\int_{\R_y^d}\int_{\R_x^d}  \big|u_\eps(s,x)-u_\eps(s,y)\big| 
\phi(\frac{x+y}{2}) J_\delta(\frac{x-y}{2})\,dx\,dy\Big]\,ds \notag \\
 & +C ||F^\prime||_{\infty} \xi\,E\Big[\int_{s=0}^t \int_{\R_y^d}\int_{\R_x^d} \big|u_\eps(s,x)-u_\eps(s,y)\big| \phi(\frac{x+y}{2})
 J_\delta(\frac{x-y}{2})\,dx\,dy\,ds\Big] \notag \\
 & + C ||F^\prime||_{\infty} \xi\,E \Big[\int_{s=0}^t \int_{\R_y^d}\int_{\R_x^d} \big|u_\eps(s,x)-u_\eps(s,y)\big| \phi(\frac{x+y}{2})
|\grad_y J_\delta(\frac{x-y}{2})|\,dx\,dy\,ds\Big] \notag \\
  &+   C\,\eps\, \int_{r=0}^t E \Big[ \int_{\R_y^d}\int_{\R_x^d} 
  \big| u_\eps(r,x) -u_\eps(r,y)\big| J_\delta(\frac{x-y}{2}) \phi(\frac{x+y}{2})\, dx\, dy\Big]\, dr \notag \\
 & + E \Big[\int_{r=0}^t\int_{|z| > 0} \int_{\R_y^d}\int_{\R_x^d} \int_{\rho=0}^1
 \beta_\xi^{\prime \prime}\Big( u_\eps(r,x) -u_\eps(r,y)+\rho\big(\eta_\eps(x,u_\eps(r,x);z)-\eta_\eps(y,u_\eps(r,y);z)\big)\Big)\notag
\\&\hspace{4.5cm}\times\big| \eta_\eps(x,u_\eps(r,x);z)
-\eta_\eps(y,u_\eps(r,y);z)\big|^2  \psi_\delta(x,y)\,d\rho\,dx \,dy\,\nu(dz) \,dr\Big], \label{eq:bv-frac}
\end{align} where we have used  $|\Delta \phi(x)| \le C \phi(x)$.

As before, one can use Cauchy-Schwartz inequality along with uniform moment estimate \eqref{uniform-estimates} to conclude 
 \begin{align}
 & C ||F^\prime||_{\infty} \xi\,E \Big[\int_{s=0}^t \int_{\R_y^d}\int_{\R_x^d} \big|u_\eps(s,x)-u_\eps(s,y)\big| \phi(\frac{x+y}{2})
 J_\delta(\frac{x-y}{2})\,dx\,dy\,ds\Big] \notag \\
 & + C ||F^\prime||_{\infty} \xi\,E \Big[\int_{s=0}^t \int_{\R_y^d}\int_{\R_x^d} \big|u_\eps(s,x)-u_\eps(s,y)\big| \phi(\frac{x+y}{2})
|\grad_y J_\delta(\frac{x-y}{2})|\,dx\,dy\,ds\Big] \notag \\
\le & C ||F^\prime||_{\infty} \big( \xi + \frac{\xi}{\delta} \big) ||\phi||_{L^{\infty}(\R^d)} \sqrt{t}. \label{esti:middle term}
 \end{align}
Next, we focus on the last term of \eqref{eq:bv-frac}. To estimate that term, we first let
 \begin{align*}
 a= u_\eps(t,x)-u_\eps(t,y) \quad \text{and} \quad 
 b=\eta_\eps(x,u_\eps(t,x);z)-\eta_\eps(y,u_\eps(t,y);z).
 \end{align*} 
 Observe that
\begin{align}
  b^2\beta_\xi^{\prime\prime}(a+\rho\,b)&=(\eta_\eps(x,u_\eps(t,x);z)-\eta_\eps(y,u_\eps(t,y);z))^2\,
  \beta_\xi^{\prime\prime}(a+\rho\,b)\notag \\
& \leq \Big(|u_\eps(t,x)-u_\eps(t,y)|^2 + K^2|x-y|^2\Big)(1\wedge | z|^2)\,\beta_\xi^{\prime\prime}(a+\rho\,b)\notag \\
&= \Big( a^2 + K^2|x-y|^2\Big) \,\beta_\xi^{\prime\prime}(a+\rho\,b)\, (1\wedge | z|^2). \label{eq:nonlocal-estim}
\end{align}
 As before (cf. \ref{estimate-a}), one can use assumption~\ref{B3} on $\eta(x,u;z)$ to conclude
    $$ 0 \le a \le (1-\lambda^*)^{-1} \big(a+ \rho b+ K|x-y|\big).$$
In view of \eqref{eq:nonlocal-estim}, we have 
\begin{align*}
 b^2 \beta''_\xi (a+\rho\, b) \le & (1-\lambda^*)^{-2} \big(a+\rho\,b+ K|x-y|\big)^2 \, \beta''_\xi(a+\rho b)
 \,(|z|^2\wedge 1) + \frac{ K|x-y|^2}{\xi}\,(|z|^2\wedge 1) \notag\\
  &\le    2(1-\lambda^*)^{-2} (a+ \rho b)^2 \beta''_\xi(a+\rho b)(|z|^2\wedge 1) +C(K,\lambda^*)
  \frac{|x-y|^2}{\xi}(|z|^2\wedge 1 )\notag\\
 &  \le \Big[2(1-\lambda^*)^{-2} C\xi +C(K,\lambda^*) \frac{|x-y|^2}{\xi}\Big](|z|^2\wedge 1 ),
\end{align*}
 and hence
\begin{align}
 & E \Big[\int_{r=0}^t\int_{|z|>0} \int_{\R_y^d} \int_{\R_x^d}\int_{\rho=0}^1  b^2 \beta''_\xi (a+\rho\, b)
 \psi_\delta(x,y)\,d\rho \,dx\,dy\, \nu(dz)\,dr\Big] \notag \\
 & \quad \le E \Big[\int_{r=0}^t\int_{|z|>0} \int_{\R_y^d} \int_{\R_x^d} 
 \Big\{2(1-\lambda^*)^{-2} C\xi +C(K,\lambda^*) \frac{|x-y|^2}{\xi}\Big\}(|z|^2\wedge 1 )
 \psi_\delta(x,y) \,dx\,dy\, \nu(dz)\,dr\Big] \notag \\
 &\qquad \le C_1\Big(\xi + \frac{\delta^2}{\xi}\Big)t\, ||\phi(\cdot)||_{L^\infty(\R^d)}.\label{esti:last term}
\end{align} 
Now we make use of \eqref{eq:approx to abosx}, \eqref{esti:middle term} to \eqref{esti:last term} 
 in \eqref{eq:bv-frac} and conclude
\begin{align}
 & E \Bigg[\int_{\R_y^d}\int_{\R_x^d}  \big| u_\eps(t,x) -u_\eps(t,y)\big|J_\delta(\frac{x-y}{2}) \phi(\frac{x+y}{2}) \,dx \,dy\Bigg] \notag \\
 & \qquad \le   E \Bigg[\int_{\R_y^d}\int_{\R_x^d} \big| u_\eps(0,x) -u_\eps(0,y)\big|J_\delta(\frac{x-y}{2}) \phi(\frac{x+y}{2}) \,dx\, dy\Bigg] \notag \\
 &  \quad + C\big(1+||F^\prime||_{\infty}\big)\int_{s=0}^t E \Big[\int_{\R_y^d}\int_{\R_x^d} \big|u_\eps(s,x)-u_\eps(s,y)\big| 
\phi(\frac{x+y}{2}) J_\delta(\frac{x-y}{2})\,dx\,dy\Big]\,ds  \notag \\
 &  \qquad \quad + C ||F^\prime||_{\infty} \big( \xi + \frac{\xi}{\delta}\big) ||\phi||_{L^{\infty}(\R^d)} \,\sqrt{t} 
  + C \big( \xi + \frac{\delta^2}{\xi}\big)\,t||\phi(\cdot)||_{L^\infty(\R^d)} + C \xi\,||\phi||_{L^1(\R^d)}.\notag
\end{align} 
A simple application of Gronwall's inequality reveals that
\begin{align}
  & E \Bigg[\int_{\R_y^d}\int_{\R_x^d}  \big| u_\eps(t,x) -u_\eps(t,y)\big|J_\delta(\frac{x-y}{2}) \phi(\frac{x+y}{2}) \,dx \,dy \Bigg]\notag \\
& \le  \exp\Big( t\,C\big(1+ ||F^\prime||_{L^\infty}\big)\Big) E \Big[\int_{\R_y^d}\int_{\R_x^d} \big| u_\eps(0,x) -u_\eps(0,y)\big|
J_\delta(\frac{x-y}{2}) \phi(\frac{x+y}{2}) \,dx\, dy\Big] \notag \\
   & +   \exp\Big( t\,C\big(1+ ||F^\prime||_{L^\infty}\big)\Big)\Big[C \Big( ||F^\prime||_{\infty} \big( \xi + \frac{\xi}{\delta}\big)\sqrt{t} 
  + \big( \xi + \frac{\delta^2}{\xi}\big)\,t \Big)||\phi||_{L^\infty(\R^d)}
  + C \xi\,||\phi||_{L^1(\R^d)}\Big].\label{eq:stochastic-final-8}
\end{align}  
Chosing $\xi = C \delta^\frac{3}{2}$ in \eqref{eq:stochastic-final-8}, we obtain
\begin{align*}
E \Big[\int_{\R_y^d}\int_{\R_x^d}  \big| u_\eps(t,x) & -u_\eps(t,y)\big|J_\delta(\frac{x-y}{2}) \phi(\frac{x+y}{2})\,dx\,dy\Big] \notag \\
 & \le   C(T) E \Big[\int_{\R_y^d}\int_{\R_x^d}  \big| u_\eps(0,x) -u_\eps(0,y)\big|
J_\delta(\frac{x-y}{2}) \phi(\frac{x+y}{2}) \,dx\, dy\Big] \notag \\
&   \hspace{4cm} + C(T) \Big( \big(\delta^\frac{3}{2}  +\sqrt{\delta}\big) ||\phi||_{L^\infty(\R^d)} +  \delta^\frac{3}{2}||\phi||_{L^1(\R^d)} \Big),
\end{align*}  for some constant $C(T)>0$, independent of $\eps$.

Now we make use of the following change of variables 
$$\bar{x}= \frac{x-y}{2}, \,\, \text{and} \,\, \, \bar{y}=\frac{x+y}{2},$$
to rewrite the above inequlity (dropping  the bar). The result is
\begin{align}
  E \Big[\int_{\R_y^d}\int_{\R_x^d}  \big| u_\eps(t,x+y) & -u_\eps(t,x-y)\big|J_\delta(y) \phi(x) \,dx \,dy\Big] \notag \\
& \le   C(T) E \Big[\int_{\R_y^d}\int_{\R_x^d}  \big| u_\eps(0,x+y) -u_\eps(0,x-y)\big|
J_{\delta}(y) \phi(x) \,dx\, dy\Big]\notag \\
& \hspace{3cm} + C(T) \Big( \big(\delta^\frac{3}{2}  +\sqrt{\delta}\big) ||\phi||_{L^\infty(\R^d)} +  \delta^\frac{3}{2}||\phi||_{L^1(\R^d)} \Big)
  \label{eq:stochastic-final-9}
\end{align}
In view of \eqref{lem:modulus-continuity-part-2} of the Lemma \ref{lem: deterministic-modulus-continuity}, we obtain
  for $r < \frac{1}{2}$
 \begin{align}
 \sup_{|y| \le \delta} E \Big[\int_{\R_x^d} & \big|u_\eps(t,x+y)  -u_\eps(t,x)\big|\phi(x)\,dx\Big] \notag \\
   & \le  C_2 \,\delta^r \sup_{0<\delta \le 1} \delta^{-\frac{1}{2}}  E \Big[\int_{\R_y^d}\int_{\R_x^d} \big|u_\eps(t,x+y)-u_\eps(t,x-y)\big|
  J_{\delta}(y)\phi(x)\,dx\,dy\Big]  \notag \\
  &\hspace{7.5cm}+ C_2 \delta^r E \Big[||u_\eps(t,\cdot)||_{L^1(\R^d)}\Big]. \label{eq:stochastic-final-10}
  \end{align} 
Again, by \eqref{lem:modulus-continuity-part-1} of the Lemma \ref{lem: deterministic-modulus-continuity} and
by \eqref{eq:stochastic-final-9}, we see that for $r=\frac12$ and $s>\frac{1}{2}$
\begin{align}
   \sup_{0<\delta \le 1} \delta^{-\frac{1}{2}}  E \Big[\int_{\R_y^d}\int_{\R_x^d} & \big|u_\eps(t,x+y)  -u_\eps(t,x-y)\big|
  J_\delta(y)\phi(x)\,dx\,dy\Big] \notag \\
    & \le  C(T) \sup_{0<\delta \le 1} \delta^{-\frac{1}{2}} E \Big[\int_{\R_y^d}\int_{\R_x^d} \big|u_\eps(0,x+y)-u_\eps(0,x-y)\big| J_\delta(y)
  \phi(x)\,dx\,dy\Big] \notag \\
  & \hspace{6.5cm} + C(T)\Big( ||\phi||_{L^\infty(\R^d)} + ||\phi||_{L^1(\R^d)} \Big) \notag \\
    & \le  C(T) \, C_1\, \sup_{|y| \le \delta}\Bigg( |y|^{-s} \, E\Big[\int_{\R_x^d}  \big|u_\eps(0,x+y)-u_\eps(0,x)\big|
  \phi(x)\,dx\Big] \Bigg)  \notag \\
  & \hspace{6.5cm}+ C(T)\Big( ||\phi||_{L^\infty(\R^d)} + ||\phi||_{L^1(\R^d)} \Big)  \notag \\
   & \le  C(T)\, E\Big[||u_0||_{B_{1, \infty}^{\mu}(\R^d)}\Big] \,||\phi||_{L^\infty(\R^d)} +
  C(T)\Big( ||\phi||_{L^\infty(\R^d)} + ||\phi||_{L^1(\R^d)} \Big).\label{eq:stochastic-final-11}
 \end{align} 
Now we combine \eqref{eq:stochastic-final-10} and \eqref{eq:stochastic-final-11} to obtain
 \begin{align*}
    \sup_{|y| \le \delta} E \Big[\int_{\R_x^d} & \big|u_\eps(t,x+y)-u_\eps(t,x)\big|\phi(x)\,dx\Big] \notag \\
   &  \le C(T)\,\delta^r\Bigg[ \bigg( E\Big[||u_0||_{B_{1, \infty}^{\mu}(\R^d)}\Big] + 1 \bigg) ||\phi||_{L^\infty(\R^d)} + ||\phi||_{L^1(\R^d)}\Bigg] 
   + C_2\, \delta^r E \bigg[||u_\eps(t,\cdot)||_{L^1(\R^d)}\bigg].
 \end{align*}
 Let $K_{R}=\{x: |x| \le R\}$. Choose $\phi \in C_c^\infty(\R^d)$ such that $\phi(x)=1$ on $K_R$. Then,
  for $r < \frac{1}{2}$, we have
 \begin{align*}
   \sup_{|y| \le \delta} E\Bigg[\int_{K_R} \big|u_\eps(t,x+y)-u_\eps(t,x)\big|\,dx\Bigg] \le C(T,R)\,\delta^r,
 \end{align*} 
which completes the proof.
\end{proof}

In view of the well-posedness results from \cite{kbm2014}, we can finally claim  the existence of entropy solutions for \eqref{eq:levy_stochconservation_laws_spatial} that satisfies the fractional $BV$ estimate in Theorem \ref{thm:fractional BV estimate}. In other words, we have the following theorem.

\begin{thm}
Suppose that the assumptions ~\ref{A2}, ~\ref{A3}, ~\ref{A4'}, ~\ref{B3}, and ~\ref{B4} hold and the initial data $u_0$ belong to the
Besov space $B^{\mu}_{1, \infty} (\R^d)$ for some $\mu \in (\frac12,1)$ and
\begin{align}
\label{eq:cond}
 E \Bigg[\norm{u_0}^p_{L^p(\R^d)} + \norm{u_0}^p_{L^2(\R^d)}  \Bigg] < \infty, \,\, \text{for $p=1,2,\cdots.$}
\end{align}
\begin{itemize}
\item [(a)] Then given initial data $u_0$, there exists an entropy solution of \eqref{eq:levy_stochconservation_laws_spatial} such that for any $t\ge0$,
\begin{align*}
 E \Big[\norm{u(t,\cdot)}^p_{L^p(\R^d)} \Big] < \infty, \,\, \text{for $p=1,2,\cdots.$}
\end{align*}
Moreover, there exists a constant $C_T^R$ such that, for any $0< t < T$,
\begin{align*}
\sup_{|y| \le \delta} E\Bigg[\int_{K_R} \big|u(t,x+y)-u(t,x)\big|\,dx\Bigg] \le C_T^R \,\delta^r,
\end{align*}
for some $r \in (0, \frac12)$ and $K_{R}:=\{x: |x| \le R\}$.
\item [(b)] Let the initial data $u_0$ only satisfies \eqref{eq:cond}. Then there exists an entropy solution of \eqref{eq:levy_stochconservation_laws_spatial} such that for any $t\ge0$,
\begin{align*}
 E \Big[\norm{u(t,\cdot)}^p_{L^p(\R^d)} \Big] < \infty, \,\, \text{for $p=1,2,\cdots.$}
\end{align*}
\end{itemize}
\end{thm}


\begin{thebibliography}{99}


\bibitem{vallet}
C.~ Bauzet, G.~ Vallet, and P.~ Wittbold.
\newblock The Cauchy problem for a conservation law with a multiplicative stochastic perturbation. 
\newblock{\em Journal of Hyperbolic Differential Equations,} 9 (2012), no 4, 661-709. 


\bibitem{kbm2014}
I. H. Biswas, K. H. Karlsen, and A.~ K. Majee.
\newblock Conservation laws driven by L\`{e}vy white noise.
\newblock{\em Journal of Hyperbolic Differential Equations} (to appear), 2015.


\bibitem{biswas-majee 2013}
I. H. Biswas and A.~K. Majee.
\newblock Stochastic conservation laws: weak-in-time formulation and strong entropy condition,
\newblock{\em J. Funct. Analysis}, 267 (2014), no. 7, 2199-2252. 

\bibitem{perthame}
F.~ Bouchnut and B.~ Perthame.
\newblock Kru\v{z}kov's estimates for scalar conservation laws revisited. 
\newblock{\em Trans. Amer. Math. Soc.} 350 (1998), 2847-2870. 



\bibitem{chen2005}
G.-Q.~ Chen  and K.~H.~ Karlsen.
\newblock Quasilinear anisotropic degenerate parabolic equations with time-space dependent diffusion coefficients. 
\newblock{\em Comm. Pure Appl. Anal.} 4 (2005), 241-266.

\bibitem{Chen-karlsen2012}
G.~Q.~ Chen, Q.~ Ding, and K.~ H.~ Karlsen.
\newblock On nonlinear stochastic balance laws.
\newblock{\em Arch. Ration. Mech. Anal.} 204 (2012), no 3, 707-743.

\bibitem{cockburn}
B.~ Cockburn and G.~ Gripenberg.
\newblock Continuous dependence on the nonlinearities of solutions of degenerate parabolic equations. 
\newblock{\em J.  Differential Equations}, 151 (1999), 231-251.

\bibitem{dafermos}
C.~M.~ Dafermos.
\newblock {\em Hyperbolic conservation laws in continuum physics}, volume 325
  of {\em Grundlehren der Mathematischen Wissenschaften [Fundamental Principles
  of Mathematical Sciences]}.
\newblock Springer-Verlag, Berlin, 2000.


\bibitem{Vovelle2010}
A.~ Debussche and J.~ Vovelle. 
\newblock Scalar conservation laws with stochastic forcing. 
\newblock{\em J. Funct. Analysis}, 259 (2010), 1014-1042. 


\bibitem{xu}
Z.~Dong and T.~G. Xu.
\newblock One-dimensional stochastic {B}urgers equation driven by {L}{\'e}vy
  processes.
\newblock {\em J. Funct. Anal.}, 243(2):631--678, 2007.


  




\bibitem{nualart:2008}
J.~ Feng and D.~Nualart.
\newblock Stochastic scalar conservation laws.
\newblock {\em J. Funct. Anal.}, 255(2):313--373, 2008.

\bibitem{godu}
E.~ Godlewski and P.-A.~ Raviart.
\newblock {\em Hyperbolic systems of conservation laws}, volume 3/4 of {\em
  Math{\'e}matiques \& Applications (Paris) [Mathematics and Applications]}.
\newblock Ellipses, Paris, 1991



\bibitem{kal-resibro}
K.~H.~ Karlsen and N.~H.~ Risebro.
\newblock On the uniqueness and stability of entropy solutions of nonlinear degenerate parabolic equations with rough coefficients. 
\newblock{\em Discrete Contin. Dyn. Syst.} 9 (2003), 1081-1104. 


\bibitem{KIm2005}
J.~U.~ Kim.
\newblock On a stochastic scalar conservation law, 
\newblock{Indiana Univ. Math. J.}  52 (1) (2003) 227Ð256. 


 \bibitem{kruzkov}
S.~ N.~ Kruzkov
\newblock First order quasilinear equations with several independent variables.
\newblock {\em Mat. Sb. (N.S.)} , 81(123): 228-255, 1970.




\bibitem{peszat}
S.~Peszat and J.~Zabczyk.
\newblock {\em Stochastic partial differential equations with {L}{\'e}vy
  noise}, volume 113 of {\em Encyclopedia of Mathematics and its Applications}.
\newblock Cambridge University Press, Cambridge, 2007.
\newblock An evolution equation approach.








%
%
\bibitem{Volpert}
A.~I. Vol'pert.
\newblock Generalized solutions of degenerate second-order quasilinear
  parabolic and elliptic equations.
\newblock {\em Adv. Differential Equations}, 5(10-12):1493--1518, 2000.


\bibitem{Sinai1997}
E. ~Weinan, K. ~Khanin, A. ~Mazel, and Ya. ~Sinai.
\newblock Invariant measures for Burgers equation
with stochastic forcing.
\newblock{Annals of Math,} 151(2000), 877-960.




\end{thebibliography}
\end{document}